\documentclass[12pt,reqno]{amsart}

\newtheorem{lemma}{Lemma}
\newtheorem{definition}[lemma]{Definition}
\newtheorem{proposition}[lemma]{Proposition}
\newtheorem{theorem}[lemma]{Theorem}
\newtheorem{remark}[lemma]{Remark}

\newtheorem{corollary}[lemma]{Corollary}
\newtheorem{example}[lemma]{Example}

\usepackage{amssymb,amsmath,amsthm}
\usepackage{epsfig,amssymb,amsmath,amsthm,graphicx,psfrag}
\usepackage[all]{xy}
\usepackage{enumerate}
\usepackage{color}

\begin{document}
\title{String and band complexes over string almost gentle algebras}

 \author{Andr\'es Franco*, Hern\'an Giraldo*, and Pedro Rizzo*.}

\date{\today}

\begin{abstract}
We give a combinatorial description of a family of indecomposable objects in the  bounded  derived categories of a new class of algebras: string almost gentle algebras. These indecomposable objects are, up to isomorphism, the string and band complexes introduced by V. Bekkert and H. Merklen in \cite{Be-Me}. With this description, we give a necessary and sufficient condition for a given string complex to have infinite minimal projective resolution and we extend this condition for the  case of string algebras. Using this characterization we establish a sufficient condition for a  string almost gentle algebra (or a string algebra) to have infinite global dimension.\\
\end{abstract}
\maketitle


\vskip5pt
\noindent
 
\vskip5pt
\noindent
 * Instituto de Matem\'aticas, Universidad de Antioquia, Medell\'in-Colombia

 e-mails: andres.francol@udea.edu.co, pedro.hernandez@udea.edu.co, hernan.giraldo@udea.edu.co.

\section{Introduction}

The derived categories of the module category of finite-dimensional algebras were introduced in Representation Theory by D. Happel in \cite{Ha} in 1987 and they have been very useful for the development of the theory. Despite their importance, just for a few classes of algebras, the structure of their derived categories, in particular the description of the indecomposable objects, has been known. 

The theory of almost split sequences was extended with
success to the study of the derived category of an algebra of finite global dimension, where the notion of almost split sequence gave origin to the notion of almost split triangles. In \cite{Ha}, D. Happel proved that the derived category of an algebra of finite global dimension has almost split triangles and he gave a complete description in the case of hereditary algebras.

In general, it is not easy to describe the indecomposable objects in the derived category of an algebra. In \cite{Be-Me}, the authors solved this problem when the algebra is gentle. Recent work in this direction has been done by H. Giraldo and J. A.  V\'elez-Marulanda in \cite{Gi-Ve}, where they gave a description of a class of indecomposable object for the case of certain algebra of dihedral type. Also, G. Bobi\'nski in \cite{Bob}, and K.K. Arnesen and
Y. Grimeland in \cite{Ar-Gri} studied the structure of the derived categories of  finite-dimensional algebras. The first paper describes the almost split triangles in $K^b(\text{pro} \ A)$ for a gentle algebra $A$. In the second paper the authors classified the Auslander-Reiten components of
$K^b(\text{pro} \ A)$, where $A$ is a cluster tilted algebra of type $\widetilde{A}$.\\

On the other hand, P. Bergh, Y. Han and D. Madsen in  \cite{Be-Han-Ma} (see also \cite{Ha-Za}) gave a sufficient condition for a monomial algebra to have infinite global dimension. In this paper we will give an analogous sufficient condition for  the cases of both string almost gentle algebras and string algebras using different combinatorial techniques.\\

This work is organized as follows. In Section \ref{Sec-Preliminaries}, we fix some notations and present some background material about well-known results in derived categories and  Bondarenko's category. 

Then, in Section \ref{Sec:SAGalgebras}, we begin by recalling the definition of a special biserial algebra, of a string algebra and the definition of an almost gentle algebra given by E. Green and S. Schroll in \cite{Gr-Sc} in order to study a new class of algebras: the string almost gentle algebras. We state some preliminary results about these algebras and then we adapt the construction of a functor made in \cite{Be-Me} from the category $\mathfrak{p}(A)$ of perfect complexes to the Bondarenko's category. This functor will allow us to prove that string and band complexes (in the sense of V. Bekkert and H. Merklen, \cite{Be-Me}) are indecomposable objects in the bounded derived category of a string almost gentle algebra. Also, we show with an example that, in the case of a string almost gentle algebra, there are indecomposable complexes that do not correspond to a string or a band complex. 

In the last subsection of Section \ref{Sec:SAGalgebras}, we recall the definition of the global dimension of an algebra and we give, as a consequence of the previous work, a sufficient condition for a string almost gentle algebra to have infinite global dimension. This condition is analogous to that given in \cite{Be-Han-Ma} (see also \cite{Ha-Za}), but in this work we use different combinatorial techniques.

In Section \ref{Sec:main-theorem-SAG}, we state and prove the Main Theorem for the string almost gentle case. That is, we show that, in this case,  string and band complexes are indecomposable objects in the bounded derived category and we extend this result for some classes of string algebras.

Finally, in Section \ref{Sec:periodic-complexes-for-string-case}, we develop similar combinatorial techniques for the case of string algebras in order to give a characterization of string complexes with infinite minimal projective resolution. These complexes will be called \textit{periodic string complexes}. As a consequence of this characterization, we complete this work by stating a sufficient condition for a string algebra to have infinite global dimension.

\section{Preliminaries}\label{Sec-Preliminaries}

The definitions and results in this section are standard. For the convenience of the reader we recall here the introductory material from \cite{Be-Me}.

Let $A$ be a finite-dimensional algebra of the form $kQ/I=k(Q, I)$ over an algebraically closed field $k$, where $I$ is an admissible ideal of $kQ$  and $Q$ is a finite quiver. As usual, we denote by $Q_0$ (resp. by $Q_1$) the set of vertices (resp. the set of arrows) of $Q$. Let $A-\text{mod}$ be the category of finitely generated left $A-$modules.

We denote by $e_i$ the trivial path at vertex $i\in Q_0$ and by $P_i=Ae_i$ the corresponding indecomposable projective $A$-module.  We will denote by $\textbf{Pa}$ the set of all paths of $k(Q,I)$, that is, the set of all paths of $Q$ that are outside $I$ and by $\textbf{Pa}_{\geq l}$ (resp. by $\textbf{Pa}_{> l}$) the set of all paths in $\textbf{Pa}$ of length greater than or equal to a fixed non-negative integer $l$ (resp. greater than $l$).

Every element of $A=k(Q,I)$ can be represented uniquely by a linear combination of elements in $\textbf{Pa}$, and hence we can assume that $\textbf{Pa}$ is a basis for $A$.

The set $\textbf{M}$ of \textit{maximal paths} in $\textbf{Pa}$ has a very important role throughout this theory. A path $w=w_1\cdots w_n$ in  $k(Q,I)$ is \textit{maximal} in $k(Q,I)$ if for all arrows $a,b \in Q_1$, we have that  $aw$ and $wb$ are not paths in $k(Q,I)$. Furthermore, a nontrivial path $w$ of $Q$ is in $\textbf{Pa}$ if and only if it is a sub-path of a maximal path  $\widetilde{w}$ that is not in $I$ (that is, of an element of $\textbf{M}$). This maximal path has the form $\widetilde{w}=\hat{w}w\bar{w}$ with $\hat{w}, \bar{w}\in \textbf{Pa}$. For future reference, the path $\hat{w}\in\textbf{Pa}$ is called the \textit{left completion of $w$}. \\

We denote by $D(A)$ (resp. $D^-(A)$ or $D^b(A)$) the derived category of $A-\text{mod}$ (resp. the derived category of right bounded complexes of $A-\text{mod}$ or the derived category of bounded complexes of $A-\text{mod}$); by $C^b(\text{pro} \ A)$ (resp. $C^-(\text{pro} \ A)$ or $C^{-,b}(\text{pro} \ A)$) the category of bounded  complexes of projectives (resp. of right bounded  complexes of projectives or right bounded complexes of projectives with bounded cohomology).\\

Also, we denote by $K^b(\text{pro} \ A)$ (resp. $K^-(\text{pro} \ A)$ or $K^{-,b}(\text{pro} \ A)$) the corresponding homotopy categories to $C^b(\text{pro} \ A)$ (resp. $C^-(\text{pro} \ A)$ or $C^{-,b}(\text{pro} \ A)$).

By $\mathfrak{p}(A)$ we denote the full subcategory of $C^b(\text{pro} \ A)$ defined by the  complexes of projectives such that the image of every differential map is contained in the radical of the corresponding projective module.

Any projective complex is the sum of two complexes: one complex in $\mathfrak{p}(A)$ and another one isomorphic to the zero object in the derived category (because all differential maps are 0's or isomorphisms). Hence, we can assume that all the complexes we deal with are in $\mathfrak{p}(A)$.

It is well known that $D^b(A)$ is equivalent to $K^{-,b}(\text{pro} \ A)$. We state this in the following theorem (for the proof, see \cite{KoZi}, Prop 6.3.1, p. 113).

\begin{theorem}\label{Th:Konig-Zimmermann}
	$D^-(A)$ is equivalent to $K^{-}(\text{pro} \ A)$. The image of $D^b(A)$ under this equivalence is the homotopy category $K^{-,b}(\text{pro} \ A)$ of right bounded complexes of projectives with bounded cohomology.
\end{theorem}

An important consequence from  Theorem \ref{Th:Konig-Zimmermann} is that it allows a useful classification of the indecomposable objects in $D^b(A)$ (see Corollary \ref{Cor:corollary1}). Before establishing this result, we need to introduce some notations and lemmas. More details in \cite{KoZi}, \cite{Be-Me} and \cite{Wb}.

Given a complex $M^{\bullet}\in D(A)$, we denote by $P^{\bullet}_{M^{\bullet}}$ the so called \textit{projective resolution} of $M^{\bullet}$. Let $H^i(M^{\bullet})$ denote the $i$th cohomology group of $M^{\bullet}$. 

\begin{definition}
	\begin{enumerate}
		\item Let $P^{\bullet}\in C^{-,b}(\text{pro} \ A)\setminus C^b(\text{pro} \ A)$ and let $s$ be the maximal number such that $H^i(P^{\bullet})=0$ for $i\leq s$ and $P^s\neq 0$. Then $\alpha(P^{\bullet})^{\bullet}$, \emph{the brutal truncation of $P^{\bullet}$ below $s$}, is the complex given by
		
		$$\alpha(P^{\bullet})^i=\left\{ \begin{array}{lcc}
		P^i,       &   \text{if}  & i\geq s, \\
		0,         &                 & \text{otherwise.}  
		\end{array}
		\right.$$
		$$\partial_{\alpha(P^{\bullet})^{\bullet}}^i=\left\{ \begin{array}{lcc}
		\partial_{P^{\bullet}}^i,       &   \text{if}  & i\geq s, \\
		0,         &                 & \text{otherwise.}  
		\end{array}
		\right.$$
		
		\item Let $P^{\bullet}\in C^{b}(\text{pro} \ A)$, $P^{\bullet}\neq 0^{\bullet}$ and let $t$ be the maximal number such that $P^i=0$ for $i<t$. Then $\beta(P^{\bullet})^{\bullet}$, \emph{the good truncation of $P^{\bullet}$ below $t$}, is the complex given by
		
		$$\beta(P^{\bullet})^i=\left\{ \begin{array}{lcc}
		P^i,       &   \text{if}  & i\geq t, \\
		\ker \partial_{P^{\bullet}}^t, & \text{if} & i=t-1,\\
		0,         &                 & \text{otherwise.}  
		\end{array}
		\right.$$
		$$\partial_{\beta(P^{\bullet})^{\bullet}}^i=\left\{ \begin{array}{lcc}
		\partial_{P^{\bullet}}^i,       &   \text{if}  & i\geq t, \\
		\iota_{\ker \partial_{P^{\bullet}}^t} & \text{if} & i=t-1 \\
		0,         &                 & \text{otherwise,}  
		\end{array}
		\right.$$ where $\iota_{\ker \partial_{P^{\bullet}}^t}$ is the inclusion map.
	\end{enumerate}
\end{definition}

There are some well known facts which will be useful in what follows (see \cite{Be-Me}).

\begin{lemma}\label{Lem:lemma1}
	Let $M^{\bullet}\in K^{-,b}(\text{pro} \ A)\setminus K^b(\text{pro} \ A)$ be an indecomposable. Then $\beta(\alpha(M^{\bullet})^\bullet)^\bullet$ is also an indecomposable in $D^b(A)$ and $M^\bullet\cong P_{\beta(\alpha(M^\bullet)^\bullet)^\bullet}^\bullet$.
\end{lemma}

If $\mathcal{C}$ is a Krull-Schmidt category, let us denote by $\text{ind}_0 \ \mathcal{C}$ the set of objects of the spectroid $\text{ind} \ \mathcal{C}$ (see \cite{GaRo}). Then we have the following.

\begin{lemma}\label{Lem:lemma2}
	There exist spectroids $\text{ind} \ \mathfrak{p}(A)$ and $\text{ind} \ K^b(\text{pro} \ A)$ of $\mathfrak{p}(A)$ and $K^b(\text{pro} \ A)$, respectively, such that $$\text{ind}_0 \ \mathfrak{p}(A)= \text{ind}_0 \ K^b(\text{pro} \ A).$$
\end{lemma}

Now, let $\overline{\mathcal{X}(A)}:=\{M^\bullet\in \text{ind}_0 \ \mathfrak{p}(A) \ | \ P_{\beta(M^\bullet)^\bullet}^\bullet \notin K^b(\text{pro} \ A) \}$ and let $\cong_{\mathcal{X}}$ be the equivalence relation defined on the set $\overline{\mathcal{X}(A)}$ by

$$M^\bullet\cong_{\mathcal{X}}N^\bullet \quad \text{if and only if} \quad P_{\beta(M^\bullet)^\bullet}^\bullet\cong P_{\beta(N^\bullet)^\bullet}^\bullet$$ in $K^{-,b}(\text{pro} \ A)$. Thus, we use the notation $\mathcal{X}(A)$ for a fixed set of representatives of the quotient set $\overline{\mathcal{X}(A)}$ by the equivalence relation $\cong_{\mathcal{X}}$.\\

From Theorem \ref{Th:Konig-Zimmermann} and Lemmas \ref{Lem:lemma1} and \ref{Lem:lemma2}, we obtain the following corollary.

\begin{corollary}\label{Cor:corollary1}
	There exist spectroids $\text{ind} \ D^b(A)$ and $\text{ind} \ \mathfrak{p}(A)$  of $D^b(A)$ and $\mathfrak{p}(A)$, respectively, such that 
	$$\text{ind}_0 \ D^b(A)= \text{ind}_0 \ \mathfrak{p}(A) \cup \{\beta(M^\bullet)^\bullet \ | \ M^\bullet \in \mathcal{X}(A) \}.$$
\end{corollary}

\begin{remark}\label{Rem:remark1}
	If $A$ has finite global dimension, then $\mathcal{X}(A)=\emptyset$ and $\text{ind}_0 \ D^b(A)= \text{ind}_0 \ \mathfrak{p}(A)$.
\end{remark}

\subsection{Bondarenko's category}
In \cite{Bo}, (see also \cite{BoDr}), the author gives a description of all indecomposable objects in the category of representations of posets, nowadays known as the Bondarenko's category. This category is essential for our work due to its connection with the derived category of a gentle algebra. More specifically, V. Bekkert and H. Merklen proved in \cite{Be-Me} that the problem of finding the indecomposable objects in the derived category of a  gentle algebra may be reduced to find the indecomposables in a matrix problem presented and solved by Bondarenko in \cite{Bo}. In a similar way to \cite{Be-Me}, we will use the Bondarenko's category to describe an important family of indecomposables in the derived categories of string almost gentle algebras.\\

We consider a linearly ordered set $\mathcal{Y}$ endowed with an involution $\sigma$ (see \cite{Bo}). We will define the category $s(\mathcal{Y},k)$, called the \textit{Bondarenko's category of representations of posets}. 

The objects of $s(\mathcal{Y},k)$ are finite square block matrices $B=\begin{pmatrix}
B_{u}^{v}
\end{pmatrix}$ with entries in $k$, where $u, v\in\mathcal{Y}$. These matrices are called \textit{representations} or $\mathcal{Y}$-\textit{matrices} and satisfy the following conditions: 
\begin{enumerate}
	\item The horizontal and vertical partitions by blocks of $B$ are compatible. That is, the number of rows in $B_u$ is equal to the number of columns in $B^u$, for every $u\in\mathcal{Y}$. Here, $B_u$ (resp. $B^u$) represents the row $u$ (resp. the column $u$) of blocks of $B$.
	\item If $u,v\in\mathcal{Y}$ are such that $\sigma(u)=v$, then all matrices in $B_u$ have the same number of rows as all matrices in $B_v$ (and, consequently, all matrices in $B^u$ have the same number of columns as all matrices in $B^v$).
	\item $B^2=0$.
\end{enumerate}

 Also notice that some blocks may be empty. For example, the zero object in $s(\mathcal{Y}, k)$ corresponds to the matrix where all the blocks are empty.\\

A morphism in $s(\mathcal{Y},k)$ from $B$ to $C$ is a block matrix $T=\begin{pmatrix}T_u^v\end{pmatrix}$, where $u, v\in\mathcal{Y}$, with entries in $k$ such that the following conditions hold:
\begin{enumerate}
	\item The horizontal (resp. vertical) partition of $T$ is compatible with the vertical (resp. horizontal) partition of $B$ (resp. $C$). (In general, if $A, D$ are two block matrices, not necessarily square matrices, we say that the horizontal partition of $A$ is compatible with the vertical partition of $D$ if the number of rows in each $A_u$ is equal to the number of columns in each $D^u$ -- so that we can multiply $DA$ \textit{by blocks} --, and similarly we define what it means when the vertical partition of $A$ is compatible with the horizontal partition of $D$).
	\item $BT=CT$.
	\item If $v<u$, then $T_u^v=0$, i.e., all blocks below the main diagonal are zero blocks, where `$<$' is the order relation in the poset $\mathcal{Y}$.
	\item If $\sigma(u)=v$, then $T_u^u=T_v^v$.
\end{enumerate}

Since the matrices $T$ are triangular, in order to have an isomorphism $T$ in the category $s(\mathcal{Y},k)$, it is necessary and sufficient that the diagonal blocks of $T$, $T_u^u$, are all invertible. 

It is clear that $s(\mathcal{Y},k)$ is an additive $k$-category. It was shown in \cite{BoDr} that finding the indecomposable objects in $s(\mathcal{Y},k)$ can be reduced to find the indecomposables of a matrix problem introduced and solved in \cite{NaRo}.

The convenient notation for writing the indecomposable objects in the category $s(\mathcal{Y}, k)$ is described in the next subsection (see \cite{Be-Me}). 
\subsection{Indecomposables in $s(\mathcal{Y},k)$}\label{subsec-indecomposables-in-Bondarenko}

Let $\mathcal{Y}$ be a linearly ordered set with an involution $\sigma$. We define a quiver $Q(\mathcal{Y})$, associated to $\mathcal{Y}$, as follows: $$Q(\mathcal{Y})_0:=\mathcal{Y}/\sigma, \quad Q(\mathcal{Y})_1:=\mathcal{Y}\times\mathcal{Y}.$$
For any $\alpha\in\mathcal{Y}$, we denote by $\overline{\alpha}$ the class of $\alpha$ modulo $\sigma$, and for $\alpha, \beta\in\mathcal{Y}$ we set $s((\alpha,\beta))=\overline{\alpha}$ and $t((\alpha,\beta))=\overline{\beta}$. Let $p_i, \ i=1,2,$ be the natural projections of $\mathcal{Y}^2$ onto $\mathcal{Y}$. For instance $p_1((\alpha,\beta))=\alpha=p_2((\alpha,\beta)^{-1})$. Here, $(\alpha,\beta)^{-1}$  denotes the \textit{formal inverse} of $(\alpha,\beta)$, which verifies  $s((\alpha,\beta)^{-1})=t((\alpha,\beta))$, $t((\alpha,\beta)^{-1})=s((\alpha,\beta))$ and $((\alpha,\beta)^{-1})^{-1}=(\alpha,\beta)$. \\ 

\begin{description}
	\item[$\mathcal{Y}$-strings and $\mathcal{Y}$-bands]
\end{description}

Let $\overline{St(\mathcal{Y})}$ be the set of words $w=w_1\cdots w_n$ such that $w_i\in Q(\mathcal{Y})_1$ or $w_i^{-1}\in Q(\mathcal{Y})_1$ for $i=1,\dots n$,  $t(w_i)=s(w_{i+1})$ and $p_2(w_i)\neq p_1(w_{i+1})$ for $i=1,\dots, n-1$. Let us define the equivalence relation $\sim_s$ on the set $\overline{St(\mathcal{Y})}$ by $$w\sim_s u \quad \text{if and only if} \quad w=u \ \text{or} \ w=u^{-1}$$ where, if $u=u_1\cdots u_n$, then $u^{-1}=u_n^{-1}\cdots u_1^{-1}$.

We denote by $St(\mathcal{Y})$ a fixed set of representatives of such words over the equivalence relation $\sim_s$ together with the words of length zero. The elements of $St(\mathcal{Y})$ are called \textit{$\mathcal{Y}$-strings}.

Now, we consider the set of words $w=w_1\cdots w_n$ that satisfy the same conditions as above and additionally verify that $s(w)=t(w)$, $w^2\in\overline{St(\mathcal{Y})}$ and $w$ is not a power of another word of length less than $w$. Let $Ba(\mathcal{Y})$ be a fixed set of representatives of such words over the equivalence relation $\sim_r$ defined by $$w\sim_r u \quad \text{if and only if} \quad u=w \ \text{or} \ u=w[j]  \ \text{or} \ u=w^{-1}$$ where $w[j]$ is a cyclic permutation of $w$, i.e., $w[j]=w_{j+1}\cdots w_nw_1\cdots w_j$, for $j=1,\dots, n-1$. The elements of $Ba(\mathcal{Y})$ are called \textit{$\mathcal{Y}$-bands}. \\

For each $\mathcal{Y}$-string $w=w_1\cdots w_n$ we define an element $B_w$ of $s(\mathcal{Y}, k)$, that is a $\mathcal{Y}$-matrix and, in turn, a representation of $Q(\mathcal{Y})$ as follows. Let us consider a $k$-vector space with basis some set of $n+1$ vectors $v_0,\dots, v_n$. For any element $y\in\mathcal{Y}$, let $M_w(\overline{y})$ the subspace spanned by the set of $v_i$'s such that $c(i)=\overline{y}$, where $c(i)=t(w_i)$ for $i>0$ and $c(0)=s(w)=s(w_1)$. Now, for each arrow $(s,t)\in Q(\mathcal{Y})_1=\mathcal{Y}\times\mathcal{Y}$, we define the linear map $M_w(s,t)$ by 
$$
M_w(s,t)(v_i)= \left\{ \begin{array}{lc}
v_{i+1},  &   \text{if} \ w_{i+1}=(s,t). \\
v_{i-1},  &  \text{if}   \ w_i^{-1}=(s,t). \\
0,          &  \text{otherwise}.   
\end{array}
\right.
$$
With this, we define $(B_w)_s^t$ as the matrix representing $M_w(s,t)$ in the fixed basis $\{v_0,\dots, v_n\}$.\\

Analogously, for each $\mathcal{Y}$-band $w=w_1\cdots w_n$ and each indecomposable polynomial $f(x)=\alpha_1+\cdots+\alpha_dx^{d-1}+x^d$ (different from $x^d$), let us define a $\mathcal{Y}$-matrix $B_{w,f}$ as follows. Let us consider a $k$-vector space with basis some set of $n\times d$ vectors $v_{ij}$, for $i=0,\dots, n-1$ and $j=1,\dots,d$. Given any $y\in\mathcal{Y}$, let $M_{w,f}(\overline{y})$ be the subspace spanned by the set of $v_{ij}$'s such that $c(i)=\overline{y}$. Now, for each arrow $(s,t)\in Q(\mathcal{Y})_1$, let us define the linear map $M_{w,f}(s,t)$ by 
$$
M_{w,f}(s,t)(v_{ij})= \left\{ \begin{array}{ll}
v_{i+1 j},                                       & \text{if} \ i\neq n-1 \ \text{and} \ w_{i+1}=(s,t). \\
v_{i-1 j},                                       & \text{if} \ i\neq 0 \ \text{and} \  w_i^{-1}=(s,t). \\
v_{0 j+1}                                      & \text{if} \ i=n-1, j\neq d \ \text{and} \ w_n=(s,t). \\
v_{n-1 j+1}                                   & \text{if} \ i=0, j\neq d \ \text{and} \ w_n^{-1}=(s,t). \\
-\sum_{r=1}^{d}\alpha_rv_{0r}      & \text{if} \ i=n-1, j=d \ \text{and} \ w_n=(s,t). \\
-\sum_{r=1}^{d}\alpha_rv_{n-1 r}  & \text{if} \ i=0, j=d \ \text{and} \ w_n^{-1}=(s,t). \\
0,                                                &  \text{otherwise}.   
\end{array}
\right.
$$
Now, we denote by $(B_{w,f})_s^t$ the matrix representing $M_{w,f}(s,t)$ in the basis $\{v_{ij}\}$.\\

From Proposition 1 in \cite{Bo}  (section 1, p. 59) and Theorem 3 in \cite{BoDr} (section 2, p 2521), we obtain the following theorem, which will be fundamental in the next sections.

\begin{theorem}\label{Th-indecomposables-in-Bondarenko}
	Let $\mathcal{Y}=(\mathcal{Y}, \sigma)$ be a linearly ordered set with involution. Then $$\text{ind}_0 \ s(\mathcal{Y}, k)=\{B_u \mid u\in St(\mathcal{Y})\}\dot{\cup}\{B_{v,f} \mid v\in Ba(\mathcal{Y}), f\in\text{Ind}\ k[x]\}$$
	where $k[x]$ is the $k$-algebra of polynomials in the variable $x$.
\end{theorem}

\section{String almost gentle algebras}\label{Sec:SAGalgebras}

We begin by recalling the definition of a special biserial algebra and, in particular, of a string algebra.

\begin{definition}
	The algebra $A=kQ/I$ is called \textit{special biserial} if it satisfies the following conditions:
	\begin{enumerate}
		\item Any vertex of $Q$ is the starting point of at most two arrows. Any vertex of $Q$ is the ending point of at most two arrows.
		\item Given an arrow $a\in Q_1$, there is at most one arrow $b\in Q_1$ with $s(b)=t(a)$ and $ab\notin I$.
		\item Given an arrow $a\in Q_1$, there is at most one arrow $c\in Q_1$ with $t(c)=s(a)$ and $ca\notin I$.
	\end{enumerate}
\end{definition}

It is a well known fact that, for special biserial algebras, $I$  can be generated by zero relations and by commutativity relations.

\begin{definition}
	A special biserial algebra $A=kQ/I$ is called a \textit{string algebra} if, additionally, $I$ is generated by zero relations, i.e., by paths of length greater than or equal to 2.
\end{definition}

Now, we recall the definition of \textit{almost gentle algebras}, given by E. Green and S. Schroll in \cite{Gr-Sc}. Besides, we put together this definition with the corresponding one for string algebras in order to define the class of \textit{string almost gentle algebras} or \textit{SAG algebras} for short.

\begin{definition}\label{Def-almost-gentle}
	An algebra $A=kQ/I$ is called an \textit{almost gentle algebra} if 
	\begin{enumerate}
		\item $I$ is generated by paths of lengths 2.
		\item For every arrow $a\in Q_1$ there is at most one arrow $b\in Q_1$ such that $ab\notin I$ and at most one arrow $c\in Q_1$ such that $ca\notin I$.
	\end{enumerate}	
\end{definition}
Thus, an algebra is almost gentle if it is Morita equivalent to a special multiserial algebra $kQ/I$, where $I$ is generated by monomial relations of length two. 

\begin{definition}
	An algebra $A=kQ/I$ is called a \textit{string almost gentle algebra} or, simply, a \textit{SAG algebra} if it satisfies the following conditions:
	\begin{enumerate}
		\item Any vertex of $Q$ is the starting point of at most two arrows. Any vertex of $Q$ is the ending point of at most two arrows.
		\item Given an arrow $a\in Q_1$, there is at most one arrow $b\in Q_1$ with $s(b)=t(a)$ and $ab\notin I$.
		\item Given an arrow $a\in Q_1$, there is at most one arrow $c\in Q_1$ with $t(c)=s(a)$ and $ca\notin I$.
		\item $I$ is generated by paths of lengths two.
	\end{enumerate}
\end{definition}

Thus, the class of SAG algebras is the intersection of the class of string algebras with the class of almost gentle algebras. The following result about SAG algebras establishes that maximal paths intersect only in vertices.
\begin{lemma}\label{lem-3.1}
	Let $A=kQ/I$ be a SAG algebra. Then two different maximal paths in $\textbf{M}$ cannot have a common arrow.
\end{lemma}

\begin{proof}[Proof]
	Suppose that $w, w'$ are two different maximal paths and $a$ is a common arrow to both of them. Then, if $a$ is followed by an arrow $b$ in $w$ and another arrow $c$ in $w'$, we get a contradiction with the definition of SAG algebras, because we would have that $ab\notin I$ and $ac\notin I$. If $a$ is the last arrow of only one of $w$ and $w'$, we have a contradiction with the maximality of this path. Finally, if $a$ is the last arrow of both paths, we consider the previous arrow in each path. If, in fact, this arrow is common, we consider the previous one. After a finite number of steps  (since $Q$ is finite), we must find a not common arrow such that all the following are common arrows. But this is a contradiction with the definition of SAG algebras because we would have two compositions that do not lie in $I$.
\end{proof}

\begin{remark}\label{rem-almost-gentle}
	We notice that the previous lemma remains valid if we require only that the algebra is almost gentle. Thus, for almost gentle algebras maximal paths intersect only in vertices.
\end{remark}

An immediate consequence of Lemma \ref{lem-3.1} is that maximal paths for SAG algebras ({\it a fortiori}, in light of remark \ref{rem-almost-gentle}, for almost gentle algebras) are unique for each arrow.

\begin{corollary}
	Let $A=kQ/I$ be a SAG algebra. Then $Q$ cannot have a vertex which is the ending point of three, different, maximal paths in $\textbf{M}$.
\end{corollary}

\begin{proof}
	Since $A$ is a SAG algebra, then there are at most two arrows ending in a given vertex. Thus, if this vertex is the ending point of three, different, maximal paths in $\textbf{M}$, then at least two of them must share the last arrow, but this contradicts the previous lemma.
\end{proof}

\subsection{The Functor}\label{Subsec: the functor} 

This section deals with an adapted version of the construction made in \cite{Be-Me} of a functor $F$ from the category $\mathfrak{p}(A)$ to the category $s(\mathcal{Y}(A), k)$, when $A$ is a SAG algebra. This functor and its properties will allow us to prove that string and band complexes are indecomposable objects in $D^b(A)$.\\

We start with a finite complex $P^{\bullet}\in\mathfrak{p}(A)$ of length $m$, that is, a complex $P^{\bullet}$ of projective modules  such that the image of every differential map is contained in the radical of the corresponding projective module. The complex $P^{\bullet}$ has the form

$$\xymatrix{\cdots\ar[r]&0\ar[r]&P^n\ar[r]^{\partial^n}&\cdots\ar[r]&P^{n+m-1}\ \ar[r]^(0.55){\partial^{n+m-1}} \ &P^{n+m}\ \ar[r]&0\ar[r]&\cdots}$$
with $n,m\in\mathbb{Z}$.\\

If in each  $P^j$ of the complex $P^{\bullet}$, the indecomposable projective module $P_i$ appears $d_{i,j}$ times, we will use the notation $P_i^{d_{i,j}}$.  Therefore, we can write the complex $P^{\bullet}$ as

$$\xymatrix{0\ar[r]&\displaystyle\bigoplus_{i=1}^{t}P_i^{d_{i,n}}\ar[r]^(0.65){\partial^n}&\cdots\ar[r]&\ \displaystyle\bigoplus_{i=1}^{t}P_i^{d_{i,n+m-1}} \ar[r]^(0.55){\partial^{n+m-1}} \ &\ \displaystyle \bigoplus_{i=1}^{t}P_i^{d_{i,n+m}}\ar[r]&0}$$ where $|Q_0|=t$.\\

 Since every projective module is a finite direct sum of indecomposable projective modules, each morphism (differential) in the complex $P^{\bullet}$ is given by a block matrix (of size $\sum_id_{i,j}\times\sum_id_{i,j+1}$ if the differential goes from the place $j$ to the place $j+1$). Thus, each block gives the component of the morphism corresponding to each pair of indecomposables. In other words, each block corresponds to a morphism  $P_r^{d_{r,j}}\longrightarrow P_s^{d_{s,j+1}}$.\\
 
 It is well known that the paths $w\in\textbf{Pa}$  such that $s(w)=r$ and $t(w)=s$ form a basis for $\text{Hom}(P_r,P_s)$. In particular, on the category $\mathfrak{p}(A)$, we can assume that only the paths $w\in\textbf{Pa}_{\geq 1}$ are involved, since trivial paths give isomorphisms.

 If $w\in\textbf{Pa}_{\geq 1}$ is one of these paths, it defines the morphism $p(w): P_r\longrightarrow P_s$, given by  $u\mapsto v=uw$. It follows that any homomorphism from $P_r$ to $P_s$ is associated to a linear combination of these $p(w)$.\\

Recall that a nontrivial path $w$ determines a maximal path $\widetilde{w}=\hat{w}w\bar{w}$, where $\hat{w}, \bar{w}\in \textbf{Pa}$. According to Lemma \ref{lem-3.1}, this maximal path is unique for SAG algebras. Hence, $\hat{w}=u$ is a special basis element of $P_r$ such that $s(u)=s(\widetilde{w}), \quad t(u)=s(w)$. In addition $v=uw\neq 0$ is a sub-path of the maximal path  $\widetilde{u}=\widetilde{w}$ and it is an element in the basis of $P_s$. Consequently, the pair of nonzero paths $(u,v)$ with the same starting point such that $\widetilde{u}=\widetilde{v} \ (=\widetilde{w})$ and $l(v)>l(u)\geq 1$ determines a path $w\in\textbf{Pa}_{\geq 1}$ such that $v=uw$.\\

Now, the representation of the complex $P^{\bullet}$ is determined by representation of the sequence of morphisms $\partial^j$, $j=n,\ldots,n+m-1$, (and viceversa). In turn, each sequence of $\partial^j$ is given by a block matrix $A=\left(A_{r,j}^{s,j+1}\right)$, which depends on the `multiplicities' of the morphisms $p(w)$ in $\partial^j$. Precisely, each $\partial^j$ is represented by a formal sum:

 $$\partial^j=\sum_{w\in \textbf{Pa}_{\geq 1}}p(w)A_{w,j}$$ where $A_{w,j}$ denotes the block that expresses the   `multiplicities' of the morphism $p(w)$ at $\partial^j$. Let us explain this last sentence in  more detail. Fix the place $j$ of the complex $P^{\bullet}$. The component of $\partial^j$ going from $P_r^{d_{r,j}}$ to $P_s^{d_{s,j+1}}$ is represented by a matrix (a block)
$$A_{r,j}^{s,j+1}\in \text{Mat}\left(d_{r,j}\times d_{s,j+1};k( \langle p(w_1),\dots, p(w_l)\rangle)\right),$$ 
where the  $w_i$'s are parallel (nontrivial) paths from  $r$ to $s$ and $k( \langle  p(w_1),\dots, p(w_l)\rangle)$ is the $k$-vector space with basis  $\{p(w_1),\dots, p(w_l)\}$. It is then clear that $A_{r,j}^{s,j+1}$ can be written uniquely as $$A_{r,j}^{s,j+1}=\sum_{i=1}^{l}p(w_i)A_{w_i,j},$$ where $A_{w_i,j}\in\text{Mat}\left(d_{r,j}\times d_{s,j+1}; k\right)$.\\

\textbf{The poset $\mathcal{Y}(A)$}\\

The definition of the poset $\mathcal{Y}(A)$ is, in a nutshell, the product of two posets: the first one corresponds to the paths and the second one to the places  $j$  in the complex $P^{\bullet}$. More precisely, for each maximal path $m\in\textbf{M}$ we define a poset $\mathcal{Y}_m$ as the set of all sub-paths  $u$ of $m$ such that $s(u)=s(m)$ and ordered by their length. Thus, we define the poset  $\mathcal{Y}(A)$ as $$\mathcal{Y}(A):=\left(\displaystyle\bigcup_{m\in \textbf{M}} \mathcal{Y}_m\right)\times\mathbb{Z},$$
where the first component is an ordered disjoint union (we put on $\textbf{M}$ a fixed linear order). Now $\mathcal{Y}(A)$ is ordered anti-lexicographically, that is 
$$[u,i]<[v,j] \quad \text{iff} \quad i<j \ \text{or} \ (i=j \ \text{and} \ \widetilde{u}<\widetilde{v}) \ \text{or} \ (i=j, \widetilde{u}=\widetilde{v} \ \text{and} \ l(u)<l(v)).$$
It should be noticed that it is possible that a trivial path $e_r$ belongs to two different maximal paths. If this happens, the two occurrences of $e_r$ must be regarded as different.
Now, we endow the poset $\mathcal{Y}(A)$ with an involution $\sigma$ (see \cite{Be-Me}). Let us first define an equivalence relation on $\mathcal{Y}(A)$ by
$$[u,i]\cong_\sigma [v,j] \quad \text{if and only if} \quad i=j \ \ \text{and} \ \ t(u)=t(v).$$

 It follows that, for each fixed place $j$ in the complex, $[u,j]\cong_\sigma [v,j]$ if and only if the paths $u,v\in \displaystyle\bigcup_{m\in \textbf{M}}\mathcal{Y}_m$ correspond to the same indecomposable projective. Therefore, we define $\sigma$ as being the involution corresponding to $\cong_\sigma$.\\

Lemma \ref{lem-3.1} guarantees the well definition of $\sigma$. In fact, since there are no more than two nontrivial paths ending at a given vertex, we define: either $\sigma(u)=u$, if there is only one path $u$ ending on the vertex or $\sigma(u)=v$, if there are two ending on it.

\begin{corollary}\label{cor3.3}
	Let $A=kQ/I$ be a SAG algebra and let $u,v$ be two nontrivial paths from $\displaystyle\bigcup_{m\in \textbf{M}}\mathcal{Y}_m$ such that $\widetilde{u}\neq \widetilde{v}$ and $t(u)=t(v)$. Then,  $u\neq v$ and $u,v$ cannot have a common arrow. 
\end{corollary}

For the next proposition, we will use the following notation: let $u,v,w$ be three nontrivial paths from $\displaystyle\bigcup_{m\in \textbf{M}}\mathcal{Y}_m$, which are, respectively, sub-paths of the maximal paths $\widetilde{u}, \widetilde{v}, \widetilde{w}$ of $\textbf{M}$ and such that $t(u)=t(v)=t(w)$.

\begin{proposition}\label{prop-3.4}
	Let $A=kQ/I$ be a SAG algebra. Then, at least two of the paths $u, v, w$ must be equal.
\end{proposition}

\begin{proof}[Proof]
	If the three paths $u, v, w$ are pairwise different, they differ (pairwise) in at least one arrow, so the same is true for the maximal paths $\widetilde{u}, \widetilde{v}, \widetilde{w}$. Thus $\widetilde{u}, \widetilde{v}, \widetilde{w}$ are pairwise different. Now, since $t(u)=t(v)=t(w)$ and $A$ is a SAG algebra, then at least two of them must share the last arrow, but this contradicts Lemma \ref{lem-3.1} and Corollary \ref{cor3.3}.
\end{proof}

\begin{remark}\label{Rem-involution-SAG}
	In the case of SAG algebras, by definition we can have, locally, the following two situations (which, clearly, cannot occur if the algebra is gentle):
	\begin{enumerate}
		\item $\xymatrix{\cdot\ar[r]^{a}&\cdot \ar@<0.7ex>[r]^{b} \ar@<-0.5ex>[r]_c &\cdot}$ with $ab=0$ and $ac=0$.
		\item $\xymatrix{\cdot\ar@<0.7ex>[r]^{a} \ar@<-0.5ex>[r]_b &\cdot \ar[r]^c&\cdot}$ with $ac=0$ and $bc=0$.
	\end{enumerate}
	In the first situation, we will have three maximal paths that intersect at vertex $t(a)=s(b)=s(c)$ with three sub-paths of these maximal paths, two of them trivial paths,  ending at this vertex, namely: $e_{s(b)}, e_{s(c)}$ and another path of the form $u=u_1\cdots u_ra$, for some arrows $u_1,\dots, u_r$. Thus, in order to have a well defined involution in this case, we stipulate that $\sigma\left(\left[e_{s(b)},j\right]\right)=\left[e_{s(c)},j\right]$, and $\sigma\left(\left[u,j\right]\right)=\left[u,j\right]$, for all $j\in\mathbb{Z}$.
	
	In the second situation, there are three sub-paths of three maximal paths, respectively, ending at vertex $t(a)=t(b)=s(c)$, namely: $e_{s(c)}$ and two paths of the form $u=u_1\cdots u_ra$ and $v=v_1\cdots v_lb$, for some arrows $u_1,\dots, u_r, v_1,\dots, v_l$. In this case, we define the involution as $\sigma\left(\left[u ,j\right]\right)=\left[v, j\right]$ and $\sigma\left(\left[e_{s(c)}, j\right]\right)=\left[e_{s(c)}, j\right]$, for all $j\in\mathbb{Z}$.
\end{remark}

Now, for each complex $P^{\bullet}\in\mathfrak{p}(A)$ we must define $F(P^{\bullet})\in\mathit{s}(\mathcal{Y}(A),k)$. In this case, recall that the  nonzero blocks of $F(P^{\bullet})$ appear only at places $\left([u,j], [v,j+1]\right)$ corresponding to nontrivial paths  $w$ determined by the pair $(u,v)$ by means of the relations $$\widetilde{w}=uw\bar{w}, \ \ v=uw.$$

Thus, we define these nonzero blocks as $$F(P^{\bullet})_{[u,j]}^{[v,j+1]}:=A_{w,j}$$

We observe that the blocks outside the places  $(j,j+1)$ (that is, the diagonal  above the main diagonal) are zero blocks. Besides, the condition that all products $\partial^j\partial^{j+1}$ are equal to zero is translated as the requirement that all products of consecutive blocks are equal to zero or, equivalently, that the big matrix has square equal to zero.\\

Now, a morphism $\varphi^{\bullet}:P^{\bullet}\longrightarrow \widetilde{P}^{\bullet}$ in $\mathfrak{p}(A)$ is, at each place $j$, a homomorphism from the projective $P^j$ to the projective $\widetilde{P}^j$, consequently, it is a block matrix between direct sums of indecomposables.

$$
\xymatrix{P^\bullet\ar[d]_{\varphi^\bullet}:&\cdots\ar[r]&P^j\ar[d]_{\varphi_j}\ar[r]^{\partial^j}&P^{j+1}\ar[d]^{\varphi_{j+1}}\ar[r]&\cdots\\
	\widetilde{P}^\bullet:                 & \cdots\ar[r]&\widetilde{P}^j\ar[r]_{\widetilde{\partial}^j}&\widetilde{P}^{j+1}\ar[r]&\cdots}
$$

Thus, as we did with the differentials, if we denote by $\phi_{w,j}$ the blocks in $\varphi_j$, then we obtain from the condition $\varphi_j\widetilde{\partial}^j=\partial^j\varphi_{j+1}$, under the notations above, that:

\begin{equation}\label{Eq-F-of-a-morphism}
\sum_{w_1\in\textbf{Pa}, w_2\in\textbf{Pa}_{\geq 1} : w=w_1w_2}p(w)\phi_{w_1,j}\widetilde{A}_{w_2,j}=\sum_{w_3\in\textbf{Pa}_{\geq 1}, w_4\in\textbf{Pa} : w'=w_3w_4}p(w')A_{w_3,j}\phi_{w_4,j+1}
\end{equation}

Finally, the functor $F$ will be defined in morphisms as the matrix associated to $\varphi^\bullet$. Summarizing, we have:

\begin{definition}\label{def-functor}
	Let $F:\mathfrak{p}(A)\longrightarrow s(\mathcal{Y}(A), k)$ be the functor defined as follows:
	
	In objects, the correspondence is given by
	$$F(P^{\bullet})_{[u,i]}^{[v,j]}=\left\{ \begin{array}{ll}
	A_{w,i} \ ,  &  \text{if} \ \ j=i+1, \ v=uw, \ w\in\textbf{Pa}_{\geq 1}, \\
	0,     &  \text{otherwise.}   
	\end{array}
	\right.$$
	
	where the block $F(P^{\bullet})_{[u,i]}$ (resp. $F(P^{\bullet})^{[u,i]}$) has $d_{t(u),i}$ rows (resp. columns).
	
	In morphisms, the correspondence is
	
	$$F(\varphi^{\bullet})_{[u,i]}^{[v,j]}=\left\{ \begin{array}{ll}
	\phi_{w,i} \ ,  &  \text{if} \ \ i=j, \ v=uw, \ \widetilde{u}=\widetilde{v}, \ \text{and} \  w\in\textbf{Pa}, \\
	0,     &  \text{otherwise.}   
	\end{array}
	\right.$$
	This gives a morphism of $s(\mathcal{Y}(A),k)$.
\end{definition}


%
We have an analogous to Lemma 4 in \cite{Be-Me}, for the case of SAG algebras, as follows.

Let $\mathcal{U}$ be full subcategory of  $s(\mathcal{Y}(A), k)$ defined by  the objects of $\text{Im} F$. Then we have the following lemma, where, as usual, we use the symbol $\text{ind}_0$ to denote the set of indecomposables of the Krull-Schmidt category. The proof is similar to the one given in Lemma 4 of \cite{Be-Me}.

\begin{lemma}\label{Lemma4-SAG}
	\begin{enumerate}
		\item $\ker F=0$, where $\ker F$ is defined by the complexes $P^{\bullet}\in\mathfrak{p}(A)$ such that $F(P^{\bullet}) = 0$.
		\item If $X\cong Y$ in $\text{Im} F$, then $X\cong Y$ in $\mathcal{U}$.
	\end{enumerate}

\end{lemma}

As a consequence of this lemma, we have the following result.

\begin{corollary}\label{Cor-Lemma4-SAG}
	$\text{ind}_0 \ \mathcal{U}\subseteq  \text{ind}_0 \ \text{Im}  F$.
\end{corollary}

\begin{proof}
	First, recall that the categories $\mathcal{U}$ and $\text{Im}F$  have the same objects. Now, let $F(P^\bullet)$ an indecomposable object in $\mathcal{U}$. We must show that $F(P^\bullet)$ is indecomposable in $\text{Im}F$. If there exist complexes $Q^\bullet, R^\bullet\in\mathfrak{p}(A)$ such that  $F(P^\bullet)\cong F(Q^\bullet)\oplus F(R^\bullet)$ in $\text{Im}F$, then by Lemma \ref{Lemma4-SAG}, part 2, we have that $F(P^\bullet)\cong F(Q^\bullet)\oplus F(R^\bullet)$ in $\mathcal{U}$. But $F(P^\bullet)$ is indecomposable in $\mathcal{U}$ and hence $F(Q^\bullet)=0$ or $F(R^\bullet)=0$. This shows that $F(P^\bullet)$ is indecomposable in $\text{Im}F$.
\end{proof}
For SAG algebras, the inclusion given in Corollary \ref{Cor-Lemma4-SAG} is a proper inclusion as it is shown in the following example.
\begin{example}\label{Ex-counter-example-to-lemma4}
	Let us consider the bound quiver $(Q, I)$ given by 
	$$Q: \ \xymatrix{1 \ar[r]^{a}&2\ar@<0.7ex>[r]^{b} \ar@<-0.5ex>[r]_c &3}$$ and $I=\langle ab, ac\rangle$. Then $A=kQ/I$ is a SAG algebra, which is not gentle.
	
	Consider the complex 
	$$P^\bullet: \ \xymatrix{0\ar[r]&P_1\ar[r]^{\partial^1}&P_2\ar[r]^(0.4){\partial^2}&P_3\oplus P_3\ar[r]&0}$$ 
	where $\partial^1=p(a)$ and $\partial^2=\begin{pmatrix}
	p(b) & p(c)
	\end{pmatrix}$
	
	Let $\varphi^\bullet\in\text{Hom}_{\mathfrak{p}(A)}(P^\bullet, P^\bullet)$.
	
	$$\xymatrix{P^\bullet\ar[d]_{\varphi^\bullet}:&0\ar[r]&P_1\ar[d]_{\varphi_1}\ar[r]^{\partial^1}&P_2\ar[d]_{\varphi_2}\ar[r]^(0.4){\partial^2}&P_3\oplus P_3\ar[d]_{\varphi_3}\ar[r]&0\\
		P^\bullet:                                 &0\ar[r]&P_1\ar[r]_{\partial^1}                          &P_2\ar[r]_(0.4){\partial^2}                           &P_3\oplus P_3\ar[r]                          &0}
	$$ 
	Then we have that $\varphi^\bullet=(\varphi_1, \varphi_2, \varphi_3)$, where $\varphi_1=p(e_1)\begin{pmatrix}\lambda	\end{pmatrix}$, $\varphi_2=p(e_2)\begin{pmatrix}\alpha\end{pmatrix}$ and $\varphi_3=p(e_3)\begin{pmatrix}
	\beta_1 & \beta_2\\ \beta_3 & \beta_4
	\end{pmatrix}$, with $\lambda, \alpha, \beta_1, \beta_2, \beta_3, \beta_4\in k$.
	
	From $\varphi_1\partial^1=\partial^1\varphi_2$, it follows that $\lambda p(a)=\alpha p(a)$, and hence $\alpha=\lambda$. Also, from $\varphi_2\partial^2=\partial^2\varphi_3$ we obtain
	\begin{eqnarray*}
		\lambda p(b) & = & \beta_1 p(b)+\beta_3 p(c) \\
		\lambda p(c) & = & \beta_2 p(b)+ \beta_4 p(c)
	\end{eqnarray*}
	Therefore, $\beta_1=\lambda$, $\beta_3=0$, $\beta_2=0$ and $\beta_4=\lambda$. Hence, we have $\varphi_1=p(e_1)\begin{pmatrix}\lambda	\end{pmatrix}$, $\varphi_2=p(e_2)\begin{pmatrix}\lambda\end{pmatrix}$ and $\varphi_3=p(e_3)\begin{pmatrix}
	\lambda & 0\\ 0 & \lambda
	\end{pmatrix}$. This shows that $\text{End}_{\mathfrak{p}(A)}(P^\bullet)\cong k$ and so $P^\bullet$ is an indecomposable object in $\mathfrak{p}(A)$. 
	
	However, a direct calculation shows that it does not exist a generalized string or a generalized band $w$ in the bound quiver $(Q, I)$ such that $P_w^\bullet\cong P^\bullet$. This means that the complex $P^\bullet$ is an indecomposable object, which is not a string complex or a band complex (see subsection \ref{subsec:string and band complexes}). Therefore, in the case of SAG algebras, there are other than string and band complexes which are indecomposable objects in $D^b(A)$. This is a fundamental difference with the case of gentle algebras.\\
	
	Now, let us determine $F(P^\bullet)$ and $F(\varphi^\bullet)$ in order to illustrate how the functor works in this case. 
	
	In this example, we have that $\textbf{Pa}_{\geq 1}=\{a, b, c\}$ and $\textbf{M}=\{a<b<c\}$. Also, the poset $\mathcal{Y}(A)$ is given by $$\mathcal{Y}(A)=\{e_{s(a)}<a<e_{s(b)}<b<e_{s(c)}<c\}\times\mathbb{Z}$$ and the involution $\sigma$ is, for all $j\in\mathbb{Z}$,
	$$\sigma\left(\left[e_{s(a)},j\right]\right)=\left[e_{s(a)},j\right], \quad \sigma\left(\left[a,j\right]\right)=\left[a, j\right],$$
	$$\sigma\left(\left[e_{s(b)},j\right]\right)=\left[e_{s(c)},j\right], \quad \sigma\left(\left[b,j\right]\right)=\left[c, j\right].$$
	
	The differential maps are given  by $$\partial^j=p(a)A_{a,j}+p(b)A_{b,j}+p(c)A_{c,j}$$ and hence $$\partial^1=p(a)\begin{pmatrix} 1 \end{pmatrix}+p(b)\emptyset+p(c)\emptyset\quad \text{and}\quad \partial^2=p(a)\emptyset+p(b)\begin{pmatrix}1&0\end{pmatrix}+p(c)\begin{pmatrix}0&1\end{pmatrix}$$
	
	where $\emptyset$  indicates that this block matrix is empty. Then we have 
	\begin{center}
		\begin{tabular}{|c|c|}
			\hline 
			$F(P^\bullet)_{[u,i]}$ (resp. $F(P^\bullet)^{[u,i]}$)& Number of rows (resp. columns)    \\ 
			\hline
			\hline
			$[u,i]$ for $i\neq1,2,3$                                         & $d_{t(u),i}=0$                                \\
			\hline 
			$[e_{s(a)},1]$                                                       & $d_{1,1}=1$                                    \\ 
			\hline 
			$[a,1]$                                                                 & $d_{2,1}=0$                                    \\ 
			\hline 
			$[e_{s(b)},1]$                                                       & $d_{2,1}=0$                                   \\ 
			\hline 
			$[b,1]$                                                                & $d_{3,1}=0$                                   \\ 
			\hline 
			$[e_{s(c)},1]$                                                        & $d_{2,1}=0$                                   \\ 
			\hline 
			$[c,1]$                                                                & $d_{3,1}=0$                                   \\ 
			\hline 
			$[e_{s(a)},2]$                                                       & $d_{1,2}=0$                                    \\ 
			\hline 
			$[a,2]$                                                                 & $d_{2,2}=1$                                    \\ 
			\hline 
			$[e_{s(b)},2]$                                                       & $d_{2,2}=1$                                   \\ 
			\hline 
			$[b,2]$                                                                & $d_{3,2}=0$                                   \\ 
			\hline 
			$[e_{s(c)},2]$                                                        & $d_{2,2}=1$                                   \\ 
			\hline 
			$[c,2]$                                                                & $d_{3,2}=0$                                   \\ 
			\hline 
			$[e_{s(a)},3]$                                                       & $d_{1,3}=0$                                    \\ 
			\hline 
			$[a,3]$                                                                 & $d_{2,3}=0$                                    \\ 
			\hline 
			$[e_{s(b)},3]$                                                       & $d_{2,3}=0$                                   \\ 
			\hline 
			$[b,3]$                                                                & $d_{3,3}=2$                                   \\ 
			\hline 
			$[e_{s(c)},3]$                                                        & $d_{2,3}=0$                                   \\ 
			\hline 
			$[c,3]$                                                                & $d_{3,3}=2$                                   \\ 
			\hline 
		\end{tabular} 
	\end{center}
	
	and 
	$$
	\begin{array}{cc}
	F(P^\bullet)_{[e_{s(a)},1]}^{[a,2]}=A_{a,1}=\begin{pmatrix} 1\end{pmatrix},&  F(P^\bullet)_{[e_{s(b)},2]}^{[b,3]}=A_{b,2}=\begin{pmatrix} 1&0\end{pmatrix}, \\ F(P^\bullet)_{[e_{s(c)},2]}^{[c,3]}=A_{c,2}=\begin{pmatrix}0 & 1\end{pmatrix},    &   
	\end{array}
	$$ 	
	
	and the other cases for $F(P^\bullet)_{[u,i]}^{[v,j]}$ are zero blocks or empty blocks. Thus, we get the following block matrix $F(P^\bullet)=\begin{pmatrix}
	F(P^\bullet)_{[u,i]}^{[v,j]}
	\end{pmatrix}$
	\begin{center}
		\begin{tabular}{c||c|c|c|c|cc|cc|}
			& $[e_{s(a)},1]$ & $[a,2]$ & $[e_{s(b)},2]$ & $[e_{s(c)},2]$ & $[b,3]$ & $[b,3]$ & $[c,3]$ & $[c,3]$ \\
			\hline
			\hline
			$[e_{s(a)},1]$  & 0            &  1         & 0           &    0           & 0             & 0          & 0             &  0            \\
			\hline
			$[a,2]$           & 0            &  0         & 0           &    0           & 0             & 0          & 0             &  0             \\
			\hline
			$[e_{s(b)},2]$ & 0            &  0         & 0           &    0           & 1             & 0          & 0             &  0            \\
			\hline                                                                                                                                                      
			$[e_{s(c)},2]$   & 0            &  0         & 0           &    0           & 0             & 0          & 0             &  1            \\
			\hline
			$[b,3]$            & 0            &  0         & 0           &    0           & 0             & 0          & 0             &  0            \\
			$[b,3]$            & 0            &  0         & 0           &    0           & 0             & 0          & 0             &  0            \\
			\hline
			$[c,3]$             & 0            &  0         & 0           &    0           & 0             & 0          & 0             &  0            \\
			$[c,3]$             & 0            &  0         & 0           &    0           & 0             & 0          & 0             &  0            \\
			\hline
		\end{tabular}
	\end{center}
	
	It is easy to verify that $\left(F(P^\bullet)\right)^2=0$ and thus $F(P^\bullet)\in s(\mathcal{Y}(A), k)$. The other conditions are satisfied by construction.

	Now, for each $j\in\mathbb{Z}$, we have that 
	$$\varphi_j=p(e_1)\phi_{e_1,j}+p(e_2)\phi_{e_2,j}+p(e_3)\phi_{e_3,j}+p(a)\phi_{a,j}+p(b)\phi_{b,j}+p(c)\phi_{c,j}.$$
	Since $\varphi_1=p(e_1)\begin{pmatrix}\lambda	\end{pmatrix}$, $\varphi_2=p(e_2)\begin{pmatrix}\lambda\end{pmatrix}$ and $\varphi_3=p(e_3)\begin{pmatrix}
	\lambda & 0\\ 0 & \lambda
	\end{pmatrix}$, then $\phi_{e_1,1}=\begin{pmatrix}\lambda	\end{pmatrix}$, $\phi_{e_2,2}=\begin{pmatrix}\lambda	\end{pmatrix}$ and $\phi_{e_3,3}=\begin{pmatrix}
	\lambda & 0\\ 0 & \lambda
	\end{pmatrix}$ are the only nonempty blocks. 
	
	Thus, using a similar process to the one we used to calculate the block matrix $F(P^\bullet)$, we obtain for $F(\varphi^\bullet)=\left(F(\varphi^\bullet)_{[u,i]}^{[v,j]}\right)$ the following block matrix
	
	\begin{center}
		\begin{tabular}{c||c|c|c|c|cc|cc|}
			& $[e_{s(a)},1]$ & $[a,2]$ & $[e_{s(b)},2]$ & $[e_{s(c)},2]$ & $[b,3]$ & $[b,3]$ & $[c,3]$ & $[c,3]$ \\
			\hline
			\hline
			$[e_{s(a)},1]$  & $\lambda$  &  0                        & 0               &    0               & 0                & 0              & 0               &  0              \\
			\hline
			$[a,2]$           & 0                 &  $\lambda$         & 0               &    0               & 0                & 0              & 0               &  0              \\
			\hline
			$[e_{s(b)},2]$ & 0                 &  0                       & $\lambda$ &    0               & 0                & 0              & 0               &  0              \\
			\hline                                                                                                                                                      
			$[e_{s(c)},2]$   & 0               &  0                       & 0               &    $\lambda$ & 0               & 0               & 0               &  0                \\
			\hline
			$[b,3]$            & 0                &  0                       & 0               &    0               & $\lambda$ & 0               & 0               &  0               \\
			$[b,3]$            & 0                &  0                       & 0               &    0               & 0               & $\lambda$ & 0               &  0               \\
			\hline
			$[c,3]$             & 0               &  0                       & 0               &    0               & 0               & 0               & $\lambda$ &  0               \\
			$[c,3]$             & 0               &  0                       & 0               &    0               & 0               & 0               & 0               &  $\lambda$  \\
			\hline
		\end{tabular}
	\end{center}
	\vspace{2mm}
	This shows that $\text{Hom}_{\text{Im}F}(F(P^\bullet), F(P^\bullet))\cong k$.\\
	
	Now, let $B=F(P^\bullet)$ and let $T\in\text{Hom}_{\mathcal{U}}(F(P^\bullet), F(P^\bullet))$ an element of the form 
	$$T=
	\left(
	\begin{array}{c|c|c|c|cc|cc}
	a_{11} & 0         & 0        & 0        & 0         & 0         & 0         & 0        \\
	\hline
	0       & a_{22} & 0         & 0        & 0         & 0         & 0         & 0        \\
	\hline
	0       & 0         & a_{33} & 0        & 0         & 0         & 0         & 0        \\
	\hline
	0       & 0         & 0        & a_{44} & 0         & 0         & 0         & 0        \\
	\hline
	0       & 0         & 0        & 0         & a_{55} & a_{56} & 0         & 0        \\
	0       & 0         & 0        & 0         & a_{65} & a_{66} & 0         & 0        \\
	\hline
	0       & 0         & 0        & 0         & 0         & 0         & a_{77} & a_{78} \\
	0       & 0         & 0        & 0         & 0         & 0         & a_{87} & a_{88} \\
	\end{array}
	\right)
	$$
	
	From condition (4) of the definition of morphisms in the Bondarenko's Category (see section \ref{Sec-Preliminaries}), we get $a_{33}=a_{44}$ and $\begin{pmatrix}
	a_{55} & a_{56}\\a_{65} & a_{66}
	\end{pmatrix}=\begin{pmatrix}
	a_{77} & a_{78}\\a_{87} & a_{88}
	\end{pmatrix}$. Also, from the condition $BT=TB$, we obtain $a_{11}=a_{22}:=a$, $a_{33}=a_{55}=a_{66}:=b$, and  $a_{56}=a_{65}=0$. Thus,
	
	$$T=
	\left(
	\begin{array}{c|c|c|c|cc|cc}
	a & 0 & 0 & 0 & 0 & 0 & 0 & 0 \\
	\hline
	0 & a & 0 & 0 & 0 & 0 & 0 & 0 \\
	\hline
	0 & 0 & b & 0 & 0 & 0 & 0 & 0 \\
	\hline
	0 & 0 & 0 & b & 0 & 0 & 0 & 0 \\
	\hline
	0 & 0 & 0 & 0 & b & 0 & 0 & 0 \\
	0 & 0 & 0 & 0 & 0 & b & 0 & 0 \\
	\hline
	0 & 0 & 0 & 0 & 0 & 0 & b & 0 \\
	0 & 0 & 0 & 0 & 0 & 0 & 0 & b \\
	\end{array}
	\right)
	$$
	with $a, b\in k$. This means that in $\text{Hom}_{\mathcal{U}}(F(P^\bullet), F(P^\bullet))$ there are morphisms depending on more than one scalar parameter and hence $\text{Hom}_{\text{Im}F}(F(P^\bullet), F(P^\bullet))\subsetneq\text{Hom}_{\mathcal{U}}(F(P^\bullet), F(P^\bullet))$.
\end{example}

Lemma \ref{Lemma4-SAG} implies that there is a bijective correspondence between $\text{Isom}_{\mathfrak{p}(A)}(P^\bullet, Q^\bullet)$ and $\text{Isom}_{\text{Im}F}(F(P^\bullet), F(Q^\bullet))$. The previous example shows that $$\text{Isom}_{\text{Im}F}(F(P^\bullet), F(Q^\bullet))\subsetneq \text{Isom}_{\mathcal{U}}(F(P^\bullet), F(Q^\bullet)).$$

This occurs in the SAG case since there are more morphisms (with respect to the gentle case) due to the local configuration of $(Q, I)$, in which we have three sub-paths of maximal paths ending at the same vertex (see Remark \ref{Rem-involution-SAG}). This local configuration produces more independent blocks in the matrix corresponding to a morphism $\varphi^\bullet$, as in Example \ref{Ex-counter-example-to-lemma4}.

\subsection{String and band complexes}\label{subsec:string and band complexes}

In this subsection, we give a description of certain complexes of projectives, called \textit{string complexes} and \textit{band complexes}, which are associated to generalized strings and bands, respectively. These complexes were introduced by V. Bekkert and H. Merklen in \cite{Be-Me},  classifying with them all  indecomposable objects in $\mathfrak{p}(A)$ for the gentle case. As we shall see,  in the SAG case, they continue being an important family of indecomposable objects in the category  $\mathfrak{p}(A)$. 

Most of the material introduced in this subsection follows the definitions and notations given in \cite{Be-Me}. Our idea with this revision is to facilitate access to the reader to these topics from the SAG point of view.\\

Let $A=kQ/I$ be a SAG algebra. For each arrow $a\in Q_1$, we denote by $a^{-1}$  its \textit{formal inverse}, which verifies  $s(a^{-1})=t(a)$, $t(a^{-1})=s(a)$ and $(a^{-1})^{-1}=a$. More generally, if $p=a_1\cdots a_n$ is a path in $Q$, the \textit{inverse path} of $p$ is given by $p^{-1}=a_n^{-1}\cdots a_1^{-1}$.\\

Now, a \textit{walk} $w$ (resp. a \textit{generalized walk}) of length $n>0$ is a sequence $w_1\cdots w_n$, where each $w_i$ is either of the form $p$ or $p^{-1}$, where $p$ is an arrow (resp. a path of positive length) and such that $t(w_i)=s(w_{i+1})$ for $i=1,\dots, n-1$. It is clear that $s(w)=s(w_1)$ and $t(w)=t(w_n)$. The notion of the \textit{inverse} of a walk (resp. of a generalized walk) is defined analogously as the one for a path. Thus, the passage to inverses is an involutory transformation.

A \textit{closed walk} (resp. a \textit{closed generalized walk}) is a walk $w$ (resp. a generalized walk) such that $t(w)=s(w)$. In this case, we consider its \textit{rotations} (or \textit{cyclic permutations})  denoted by $w[j]$, where $w[j]=w_{j+1}\cdots w_nw_1\cdots w_j$, for $j=1,\dots,n-1$. 

The \textit{product} or \textit{concatenation} of two walks (resp. of two generalized walks) $w=w_1\cdots w_n$ and $w'=w_1'\cdots w_n'$ is defined as the walk (resp. generalized walk) $ww'=w_1\cdots w_nw_1'\cdots w_n'$, whenever $t(w_n)=s(w_1')$. 

We consider two equivalence relations on the set of generalized walks, denoted by $\cong_s$ and $\cong_r$, defined as follows. If $u$ and $w$ are two generalized walks, then $$u\cong_s w \quad \text{if and only if}\quad u=w \ \text{or} \ u=w^{-1}.$$

If $u$ and $w$ are two closed generalized walks, then $$u\cong_r w \quad \text{if and only if}\quad u=w[j]  \ \text{or} \ u=w^{-1}$$ for some $1\leq j\leq n$, where we interpret $w[n]$ as $w$.

\begin{definition}
	A \textit{string} is a walk $w=w_1\cdots w_n$ such that $w_{i+1}\neq w_i^{-1}$ for $1\leq i <n$ and such that no sub-word of $w$ or $w^{-1}$ is in $I$. The set of all strings in $(Q, I)$ will be denoted by $St$.
\end{definition}


Now, we denote by $\overline{GSt}$ the set of all generalized walks $w=w_1\cdot w_2\cdots w_n$ satisfying

\begin{itemize}
	\item If $w_i, w_{i+1} \in \textbf{Pa}_{>0}$, then $w_iw_{i+1}\in I$.
	\item If $w_i^{-1}, w_{i+1}^{-1}\in \textbf{Pa}_{>0}$, then $w_{i+1}^{-1}w_i^{-1}\in I$.
	\item If $w_i, w_{i+1}^{-1} \in \textbf{Pa}_{>0}$ or $w_i^{-1}, w_{i+1} \in \textbf{Pa}_{>0}$, then $w_iw_{i+1}\in St$. 
\end{itemize}

We use the notation $GSt$ for a fixed set of representatives of the quotient of $\overline{GSt}$ over the equivalence relation $\cong_s$ together with all trivial paths. The elements of $GSt$ are called \textit{generalized strings}.


\begin{remark}
    In order to distinguish  generalized strings from strings, we use a dot $\cdot$ between $w_i$ and $w_{i+1}$ in a generalized string. For instance, in the quiver  $Q: \ \xymatrix{2 \ar[r]^{a} & 1 & 3\ar[l]_{b}}$ we have that $ab^{-1}$ is a string and $a\cdot b^{-1}$ is a generalized string.
\end{remark}

Given a  generalized walk $w=w_1\cdots w_n$ we define the function $\mu_w: \{0,1,\dots, n\}\longrightarrow \mathbb{Z}$, by 
$$\mu_w(0):=0$$ and 
$$\mu_w(i):=\left\{ \begin{array}{ll}
\mu_w(i-1)+1 ,  &  \text{if} \ \ w_i\in\textbf{Pa}_{>0}, \\
\mu_w(i-1)-1 ,  &  \text{if} \ \ w_i^{-1}\in\textbf{Pa}_{>0}.   
\end{array}
\right.$$
Thus, we set $\mu(w):=\min_{0\leq i\leq n}\left\{\mu_w(i)\right\}$.

This allows us to define the set $\overline{GBa}$ of all closed generalized walks $w=w_1\cdots w_n$ (i.e. $t(w_n)=s(w_1)$) such that $w^2\in\overline{GSt}$, $\mu_w(n)=\mu_w(0)=0$ and $w$ is not a power of a shorter generalized walk. We use the notation $GBa$ for a fixed set of representatives of the quotient $\overline{GBa}/\cong_r$. The elements of $GBa$ are called \textit{generalized bands}.


\begin{remark}
	It is always possible to assume that $\mu_w(0)\leq \mu_w(n)$. In fact, if $\mu_w(0)\geq\mu_w(n)$, then $\mu_{w^{-1}}(0)\leq\mu_{w^{-1}}(n)$ and $w^{-1}\cong_s w$. 
\end{remark}

Now, for every generalized string (resp. every generalized band), we will associate a finite projective complex called a \textit{string complex} (resp. a \textit{band complex}).

\begin{definition}\label{def-string-complex}
	Let  $w=w_1\cdot w_2\cdots w_n$ be a generalized string. A \emph{string complex} $P_{w}^{\bullet}$ is a complex  $\xymatrix{\cdots P_w^i\ar[r]^{\partial_w^i}&P_w^{i+1}\cdots}$ defined as follows:
	
	\begin{enumerate}
		\item The modules are given by
		\begin{equation}\label{def-string-complex-modules}
		P_w^i=\displaystyle\bigoplus_{j=0}^n\delta(\mu_w(j),i)P_{c(j)}
		\end{equation}
		where $c(0)=s(w_1)$, $c(j)=t(w_j)$ for $j>0$ and $\delta$ is the Kronecker delta. 
		
		\item The differential maps are given by  $\partial_w^i=\left(\partial_{jk}^i\right)_{1\leq j,k\leq n}$, where
		\begin{equation}\label{def-string-complex-differentials}
		\partial_{jk}^i:=\left\{ \begin{array}{ll}
		p(w_{j+1}) ,  &  \text{if} \ \ w_{j+1}\in\textbf{Pa}_{>0}, \mu_w(j)=i \ \text{and} \ k=j+1, \\
		p(w_j^{-1}),  &  \text{if} \ \ w_j^{-1}\in\textbf{Pa}_{>0} , \mu_w(j)=i \ \text{and} \ k=j-1, \\
		0,                &  \text{otherwise}.   
		\end{array}
		\right.
		\end{equation}
		\item Also, for each trivial generalized string $e_i^{\pm 1}$, let us denote by $P_{e_i^{\pm} 1}^\bullet$, the following projective complex
		$$\xymatrix{\cdots\ar[r]&0\ar[r]&P^0=P_i\ar[r]^(0.7){\partial^0}&0\ar[r]&\cdots}$$
	\end{enumerate}

\end{definition}

Let $\text{Ind} \ k[x]$ be the set of all indecomposable polynomials, except $\{x^d \mid d\geq 1\}$, with coefficients in $k$. If $f\in \text{Ind} \ k[x] $, let us denote by $F_{f(x)}$ the corresponding Frobenius matrix.

\begin{definition}\label{def-band-complex}
	Let $w=w_1\cdot w_2\cdots w_n$ be a generalized band and let $f(x)\in \text{Ind} \ k[x]$. Then $P_{w,f}^\bullet$ is the projective complex $\xymatrix{\cdots P_{w,f}^i\ar[r]^{\partial_{w,f}^i}&P_{w,f}^{i+1}\cdots}$, called a \textit{band complex} and defined as follows. The modules are given by  $$P_{w,f}^i=\displaystyle\bigoplus_{j=0}^{n-1}\delta(\mu_w(j),i)P_{c(j)}^{\deg f}$$
	
	The differential maps are given by $\partial_{w,f}^i=\left(\partial_{jk}^i\right)_{1\leq j,k\leq n}$, where
	$$\partial_{jk}^i:=\left\{ \begin{array}{ll}
	p(w_{j+1})Id_{\deg f(x)} ,  &  \text{if} \ \ w_{j+1}\in\textbf{Pa}_{>0}, \mu_w(j)=i \ \text{and} \ k=j+1, \\
	p(w_j^{-1})Id_{\deg f(x)},  &  \text{if} \ \ w_j^{-1}\in\textbf{Pa}_{>0} , \mu_w(j)=i \ \text{and} \ k=j-1, \\
	p(w_n)F_{f(x)},                 & \text{if} \ \ w_n\in\textbf{Pa}_{>0}, \mu_w(j)=i, j=n-1 \ \text{and} \ k=0,\\ 
	p(w_n^{-1})F_{f(x)},          & \text{if} \ \ w_n^{-1}\in\textbf{Pa}_{>0}, \mu_w(j)=i, j=0 \ \text{and} \ k=n-1,\\ 
	0,                &  \text{otherwise}.   
	\end{array}
	\right.$$
\end{definition}

\begin{example}
	Let $A=kQ/I$ be the algebra given by the bound quiver $$Q: \ \xymatrix{1 \ar@<0.7ex>[r]^{a} \ar@<-0.5ex>[r]_c&2\ar@<0.7ex>[r]^{b} \ar@<-0.5ex>[r]_d &3}$$ with $I=\langle ab, cd\rangle$ and let $w$ be the generalized string $w=a^{-1}\cdot c\cdot d \cdot (cb)^{-1}=w_1\cdot w_2\cdot w_3\cdot w_4$. Then, as we have noticed above, $\mu_w(0)=0, \mu_w(1)=-1, \mu_w(2)=0, \mu_w(3)=1, \mu_w(4)=0$. Thus, from \eqref{def-string-complex-modules}, we get:
	\begin{itemize}
		\item For $i=-1$, we have $P_{w}^{-1}=\displaystyle\bigoplus_{j=0}^4\delta(\mu_w(j),-1)P_{c(j)}=P_{c(1)}=P_{t(w_1)}=P_1$.
		\item For $i=0$, we have $P_{w}^{0}=\displaystyle\bigoplus_{j=0}^4\delta(\mu_w(j),0)P_{c(j)}=P_{c(0)}\oplus P_{c(2)}\oplus P_{c(4)}=P_2\oplus P_2\oplus P_1$.
		\item For $i=1$, we have $P_{w}^{1}=\displaystyle\bigoplus_{j=0}^4\delta(\mu_w(j),1)P_{c(j)}=P_{c(3)}=P_{t(w_3)}=P_3.$
	\end{itemize}
	Now, let us calculate de differential maps. According to \eqref{def-string-complex-differentials},
	$$\partial_{w}^{-1}=\left(\partial_{jk}^{-1}\right)_{0\leq j,k\leq 4}=\begin{pmatrix}
	0                  & 0 & 0                 & 0  & 0\\
	\boxed{p(a)}& 0 & \boxed{p(c)}& 0  & \boxed{0}\\
	0                  & 0 & 0                & 0  & 0\\
	0                  & 0 & 0                & 0  & 0\\
	0                  & 0 & 0                & 0  & 0
	\end{pmatrix}\longleftrightarrow\begin{pmatrix}
	p(a) & p(c) & 0
	\end{pmatrix}$$
	and 
	$$\partial_{w}^{0}=\left(\partial_{jk}^{0}\right)_{0\leq j,k\leq 4}=\begin{pmatrix}
	0      & 0 & 0    & \boxed{0}       & 0\\
	0      & 0 & 0    & 0                    & 0\\
	0      & 0 & 0    & \boxed{p(d) } & 0\\
	0      & 0 & 0    & 0                    & 0\\
	0      & 0 & 0    & \boxed{p(cb)} & 0
	\end{pmatrix}\longleftrightarrow\begin{pmatrix}
	0 \\ p(d) \\ p(cb)
	\end{pmatrix}$$
	
	Hence, the string complex associated to $w$ is 
	$$P_w^\bullet: \ \xymatrix{\cdots\ar[r] & 0\ar[r] & P_1\ar[r]^(0.3){\partial_{w}^{-1}}&P_2\oplus P_2\oplus P_1\ar[r]^(0.7){\partial_{w}^0}&P_3\ar[r]&0\ar[r]&\cdots}$$
	
\end{example}

\subsection{Special sets and preliminary lemmas}

At this point we state and prove several results, which are the analogous,  in the case of string almost gentle algebras, to Lemmas 5 and 6 in \cite{Be-Me}. They will allow us to describe a family of indecomposable objects in the bounded derived categories of a SAG algebra.\\

For this, we start by showing what is the structure of the kernel of a morphism $p(w)$ for some $w\in\textbf{Pa}_{>0}$ when $A=kQ/I$ is a SAG algebra. From the definition of a SAG algebra, we know that there are at most two arrows $a$ and $b$ such that $t(a)=t(b)=s(w)$. We also know that it cannot happen that $aw\neq 0$ and $bw\neq 0$. Thus, at least one of $aw$ or $bw$ must be zero. Suppose, without loss of generality, that $bw=0$. It can also occur that $aw=0$. This cannot happen if $A$ is gentle and so, this is a fundamental difference with the case of gentle algebras, in which kernels are always cyclic. That is, in gentle algebras, $\ker p(w)=0$ or $\ker p(w)=Ab$ (see Lemma 5 in \cite{Be-Me}).

According to this, we have the following result, which is a generalization of Lemma 5, part (1),  in \cite{Be-Me}. It allows us to study the structure of $\ker p(w)$ for some $w\in\textbf{Pa}_{>0}$, and more generally, the structure of the projective resolutions of string and band complexes for the case of SAG algebras.
\begin{lemma}\label{New-lemma5-part1}
	Let $A=kQ/I$ be a SAG algebra and let $w\in\textbf{Pa}_{>0}$. Suppose there are two arrows $a$ and $b$ such that $aw=0$ and $bw=0$. Then  $\ker p(w)=Aa\oplus Ab$.
\end{lemma}

\begin{proof}[Proof]
	
	It is clear that $Aa\oplus Ab \subseteq \ker p(w)$. For the other inclusion, let $u$ be an element of $\ker p(w)$. Then, since $A$ is a SAG algebra, $u$ must have either $a$ or $b$ as its last arrow. Hence $u\in Aa$ or $u\in Ab$.
	%
	%
\end{proof}

Lemma \ref{New-lemma5-part1} establishes that, for the case of SAG algebras, kernels can be generated for at most two arrows, in contrast with the gentle case.

It is important to notice that we need a generalization of the cyclic sets $Q_c$, $\overline{GSt}_c$, $\overline{GSt}^c$, $GSt_c$, and $GSt^c$ introduced in \cite{Be-Me}.

Let us define a new set, which generalizes the set of cyclic arrows $Q_c$. To this end, we first rewrite the set $Q_c$.

$$
\begin{array}{ll}
Q_c := \{a\in Q_1 \mid  & \exists a_m,\dots, a_1\in Q_1 \ \text{such that}\ t(a_i)=s(a_{i+1}) \ \text{for} \ i=1,\dots m, \\ & \text{where} \ a_{m+1}=a, t(a_{m+1})=s(a_1), \ \text{and} \  a_ia_{i+1}=a_{m+1}a_1=0\}.
\end{array}
$$

It is clear that this definition of $Q_c$ coincides with the definition given in \cite{Be-Me}. Now, we define the new set (see Figure \ref{fig:Qc^*} below)
$$
\begin{array}{ll}
Q_c^* := \{a\in Q_1 \mid  & \exists a_m,\dots, a_1\in Q_1 \ \text{such that}\ t(a_i)=s(a_{i+1}) \ \text{for} \ i=1,\dots m, \\ & \text{where} \ a_{m+1}=a, \ \text{and for some} \ 1\leq j\leq m+1, \ t(a_j)=s(a_1) \\ & \text{with} \  a_ia_{i+1}=a_ja_1=0\}.
\end{array}
$$

\begin{figure}[h]
	\centering
	\includegraphics[width=0.4\linewidth]{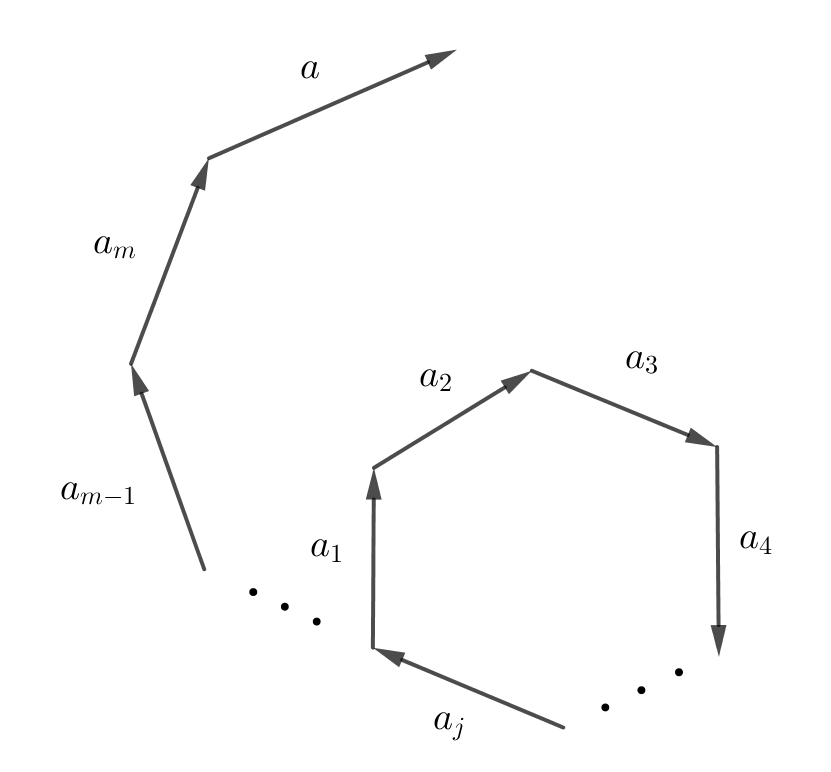}
	\caption{Definition of $Q_c^*$}
	\label{fig:Qc^*}
\end{figure}

With this definition, we have that $Q_c\subseteq Q_c^*$ and, in the case of gentle algebras, $Q_c^*=Q_c$. The elements of $Q_c^*$ are still called \textit{cyclic arrows}. 

We denote by $\overline{GSt}_{c}^*$ the set of generalized strings $w=w_1\cdot w_2\cdots w_n$ of positive length such that $\mu(w)=\mu_w(0)=0$ and there exists $a\in Q_c^*$  with $a\cdot w\in GSt$  or $a\in\ker p(w_{l}^{-1})\cap \ker p(w_{l+1})$ for some even index $l> 0$ with $\mu_w(l)=0$.\\

By $\overline{GSt}_*^c$ we also denote the set of generalized strings $w=w_1\cdot w_2\cdots w_n$ of positive length such that $\mu(w)=\mu_w(n)$ and there exists $a\in Q_c^*$ such that $w\cdot a^{-1}\in GSt$.

Thus, we have that each of these sets reduce to the corresponding one for the case of gentle algebras. Using these new special sets, we can generalize Lemma 6 in \cite{Be-Me}, with the following two lemmas.

\begin{lemma}\label{lemma7-n-converse-SAG}
	Let $w=w_1\cdot w_2\cdots w_n$ be a generalized string. If $w\in\overline{GSt}_c^*$ or $w\in\overline{GSt}_*^c$, then $P_{\beta(P_w^\bullet)^\bullet}^\bullet\notin K^b(\text{pro} \ A)$. Here, the degree in which the complex $P_w^{\bullet}$ is good truncated is precisely $\mu(w)$.
\end{lemma}
\begin{proof}
	Let $w=w_1\cdot w_2\cdots w_n$ be a generalized string. Here we will consider the case $w\in\overline{GSt}_c^*$. The case $w\in\overline{GSt}_*^c$ is similar. 
	
	If $w\in\overline{GSt}_c^*$, then $\mu(w)=0$ and there exists $a\in Q_c^*$ such that $a\cdot w=a\cdot w_1\cdot w_2\cdots w_n\in GSt$ or  $a\in\ker p(w_{l}^{-1})\cap\ker p(w_{l+1})$ for some even index $l> 0$ with $\mu_w(l)=0$.
	
	Again, we will consider the case where $a\cdot w\in GSt$. The case in which $a\in\ker p(w_{l}^{-1})\cap\ker p(w_{l+1})$ is completely analogous. Since $\mu(w)=0$, it follows that $w_1\in\textbf{Pa}_{>0}$ and hence $aw_1=0$.
	
	Since $A=k(Q, I)$ is a string almost gentle algebra, there is at most one  arrow $b$, different from $a$, such that $t(b)=t(a)=s(w_1)$. We consider the situation in which $bw_1=0$. The other cases are covered in \cite{Be-Me}. According to Lemma \ref{New-lemma5-part1}, we have  $\ker p(w_1)=Aa\oplus Ab$.
	
	Now, if $n=1$, that is, if $w=w_1$, the complex $P_w^\bullet$ is
	
	$$P_w^\bullet: \ \xymatrix{\cdots\ar[r]&0\ar[r]&P_{s(w_1)}\ar[r]^{p(w_1)}&P_{t(w_1)}\ar[r]&\cdots}$$ 
	
	Since $\ker p(w_1)=Aa\oplus Ab$ and $a\in Q_c^*$, it follows that $Aa$ has an infinite (minimal) projective resolution, so the same is true for $\ker p(w_1)$. Hence $P_{\beta(P_w^\bullet)^\bullet}^\bullet\notin K^b(\text{pro} \ A)$.
	
	If $n>1$, we must consider two cases:\\
	
	\textit{Case 1:} $w_2\in\textbf{Pa}_{>0}$. In this case, we have $\mu_w(0)=0$, $\mu_w(1)=1$, $\mu_w(2)=2$ and $w_1w_2=0$. We represent this situation as follows:
	
	$$
	\xymatrix{\cdot\ar@{~>}[rr]^{w_1}&&\cdot\ar@{~>}[rr]^{w_2}&&\cdot\ar@{.}[dl]\ar@{.}[r]&\\
		&&                                    &&}
	$$
	The possibilities for some of the next values of $\mu_w$ (in case they exist) are:
	
	$$\mu_w(3)=\left\{\begin{array}{l}
	1 \\
	3
	\end{array}\right. ,\quad \mu_w(4)=\left\{\begin{array}{l} 0\\2\\4\end{array}\right.,\quad \mu_w(5)=\left\{\begin{array}{l} 1\\3\\5\end{array}\right.$$
	Thus, in the first two places of the complex $P_w^\bullet$, we have 
	$$P_w^0=P_{s(w_1)}\oplus0\oplus0\oplus0\oplus \delta(\mu_w(4),0)P_{t(w_4)}\oplus\cdots$$
	and
	$$P_w^1=0\oplus P_{t(w_1)}\oplus0\oplus \delta(\mu_w(3),1)P_{t(w_3)}\oplus 0 \oplus \delta(\mu_w(5),1)P_{t(w_5)}\oplus\cdots$$
	The complex $P_w^\bullet$ is of the form
	$$P_w^\bullet: \ \xymatrix{\cdots\ar[r]&0\ar[r]&P_w^0\ar[r]^{\partial_w^0}&P_w^1\ar[r]^{\partial_w^1}&P_w^2\ar[r]&\cdots}$$ 
	where $$\partial_w^0 = \begin{pmatrix} 
	0         & p(w_1)   &   0       &  0         & 0 &\cdots  & 0          \\
	0         & 0           &   0       &  0         & 0 &\cdots  & 0          \\
	0         & 0           &   0       &  0         & 0 &\cdots  & 0          \\
	0         & 0           &   0       &  0         & 0 &\cdots  & 0          \\ 
	0         & 0           &   0       &  \ast     & 0 &\cdots & 0          \\
	\vdots & \vdots   & \vdots & \vdots   & 0 &\ddots & \vdots  \\
	0         & 0          &  0        & 0          & 0 & \cdots & 0
	\end{pmatrix}
	$$
	and the morphism $\ast: \delta(\mu_w(4), 0)P_{t(w_4)}\longrightarrow \delta(\mu_w(3),1)$ is either  $p(w_4^{-1})$ or zero. It is also possible that other entries in the matrix are nonzero, but the important fact is that they are not in the second column.
	
	Now, $(u_0, u_1, \dots , u_n)\in \ker \partial_w^0$ if and only if $u_0w_1=0$ together with other equations. This shows that $u_0\in\ker p(w_1)=Aa\oplus Ab$ and hence $Aa\oplus Ab$ is a direct summand of $\ker \partial_w^0$. Since $a\in Q_c^*$, it follows that $Aa$ has an infinite (minimal) projective resolution, whence the same is true for $\ker\partial_w^0$. We conclude that  $$P_{\beta(P_w^\bullet)^\bullet}^\bullet\notin K^b(\text{pro} \ A).$$
	
	\textit{Case 2:} $w_{2}^{-1}\in\textbf{Pa}_{>0}$. In this case, we have $\mu_w(0)=0$, $\mu_w(1)=1$ , $\mu_w(2)=0$ and $w_1w_2\in St$. Additionally, since $\mu(w)=0$, then $w_3\in\textbf{Pa}_{>0}$, that is $\mu_w(3)=1$ (we are assuming $w_3$ exists. If not, the reasoning is similar to the case $n=1$). The following diagram represents the situation.
	
	$$
	\xymatrix{\cdot\ar@{~>}[rrr]^{w_1}&&&\cdot\ar@{~>}[dlll]^{w_2}&&\\
		\cdot\ar@{~>}[rrr]_{w_3}&&& \ar@{.}[dl]\ar@{.}[r]                  &&\\
		&&&                                                 &&}
	$$
	The possibilities for some of the next values of $\mu_w$ are
	$$\mu_w(4)=\left\{\begin{array}{l}
	0 \\
	2
	\end{array}\right. ,\quad \mu_w(5)=\left\{\begin{array}{l} 1\\3\end{array}\right.$$
	Therefore, for $P_w^0$ and $P_w^1$ we have
	$$P_w^0=P_{s(w_1)}\oplus0\oplus P_{t(w_2)}\oplus0\oplus \delta(\mu_w(4),0)P_{t(w_4)}\oplus0\oplus\cdots$$
	and
	$$P_w^1=0\oplus P_{t(w_1)}\oplus0\oplus P_{t(w_3)}\oplus 0 \oplus \delta(\mu_w(5),1)P_{t(w_5)}\oplus 0\oplus\cdots$$
	The complex $P_w^\bullet$ is of the form
	$$P_w^\bullet: \ \xymatrix{\cdots\ar[r]&0\ar[r]&P_w^0\ar[r]^{\partial_w^0}&P_w^1\ar[r]^{\partial_w^1}&P_w^2\ar[r]&\cdots}$$ 
	where $$\partial_w^0 = \begin{pmatrix} 
	0         & p(w_1)           &   0       &  0         &\cdots  & 0          \\
	0         & 0                   &   0       &  0         &\cdots  & 0          \\
	0         & p(w_2^{-1})   &   0       & p(w_3)   &\cdots  & 0          \\
	0         & 0                   &   0       &  0         &\cdots  & 0          \\ 
	\vdots & \vdots           & \vdots & \vdots   & \ddots & \vdots  \\
	0         & 0                   &  0        & 0          &  \cdots & 0
	\end{pmatrix}
	$$
	and possibly other entries of the matrix are nonzero, but they are not in the second column.
	Now, $(u_0, u_1, u_2, \dots, u_n)\in \ker\partial_w^0$ if and only if $u_0w_1+u_2w_{2}^{-1}=0$, $u_2w_3=0$ and other equations. This implies that $u_0\in \ker p(w_1)=Aa\oplus Ab$ and $u_2\in \ker p(w_2^{-1})\cap\ker p(w_3)$. The important facts are that $Aa$ is a direct summand of $\ker\partial_w^0$ and $a\in Q_c^*$. Thus, as in the previous case, we conclude that $$P_{\beta(P_w^\bullet)^\bullet}^\bullet\notin K^b(\text{pro} \ A).$$
\end{proof}

\begin{lemma}\label{Lemma7-n-direct-SAG}
	Let $w=w_1\cdot w_2\cdots w_n$ be a generalized string. If $P_{\beta(P_w^\bullet)^\bullet}^\bullet\notin K^b(\text{pro} \ A)$, then either $w\in\overline{GSt}_c^*$ or $w\in\overline{GSt}_*^c$.
\end{lemma}
\begin{proof}
	We will show that if $w_1\in\textbf{Pa}_{>0}$, then $w\in\overline{GSt}_c^*$. By dual arguments, it can be shown that if $w_{1}^{-1}\in\textbf{Pa}_{>0}$, then $w\in\overline{GSt}_*^c$.
	
	Let us begin by assuming that $\mu(w)=0$, that is, we suppose that $w$ has as number of direct paths greater than or equal to the number of inverse paths.  We are allowed to do so, since, if $w$ has more inverse than direct paths, we consider $w^{-1}$ instead of $w$.
	If $w_1\in \textbf{Pa}_{>0}$ then $\mu_w(0)=0$, $\mu_w(1)=1$ and since $\mu(w)=0$, the possibilities for some of the next values of $\mu_w$ are 
	
	$$\mu_w(2)=\left\{\begin{array}{l}
	0 \\
	2
	\end{array}\right. ,\quad \mu_w(3)=\left\{\begin{array}{l} 1\\3\end{array}\right. ,\quad \mu_w(4)=\left\{\begin{array}{l} 0\\2\\4 \end{array}\right., \quad \mu_w(5)=\left\{\begin{array}{l} 1\\3\\5 \end{array}\right.
	$$
	Thus, for $\mu_w(l)=0$ with $l>0$, it is necessary that $l$ be an even number greater than or equal to two and less that $n$. Also, it is possible that $\mu_w(n)=0$ when $n$ is even. It follows that the complex $P_w^\bullet$ is 
	$$P_w^\bullet: \ \xymatrix{\cdots\ar[r]&0\ar[r]&P_w^0\ar[r]^{\partial_w^0}&P_w^1\ar[r]^{\partial_w^1}&P_w^2\ar[r]&\cdots}$$ 
	where $P_w^0$, $P_w^1$ are as in Lemma \ref{lemma7-n-converse-SAG} and the general form of $\ker\partial_w^0$ is either one of the following:
	\begin{enumerate}
		\item [(i)]$\ker\partial_w^0\cong \ker p(w_1)\oplus\cdots\oplus \ker p(w_{l}^{-1})\cap \ker p(w_{l+1})\oplus\cdots$, whenever $n$ is odd and $\mu_w(l)=0$ for some $2\leq l\leq n-1$.
		\item [(ii)] $\ker\partial_w^0\cong \ker p(w_1)\oplus\cdots\oplus \ker p(w_{l}^{-1})\cap \ker p(w_{l+1})\oplus\cdots\oplus \ker p(w_{n}^{-1})$, whenever $n$ is even, there is an even index $2\leq l\leq n-2$ with $\mu_w(l)=0$ and $\mu_w(n)=0$.
		\item [(iii)] $\ker\partial_w^0\cong \ker p(w_1)\oplus\cdots\oplus \ker p(w_{l}^{-1})\cap \ker p(w_{l+1})\oplus\cdots$, whenever $n$ is even, there is an even index $2\leq l\leq n-2$ with $\mu_w(l)=0$ and $\mu_w(n)>0$.
	\end{enumerate}
	Now, since $P_{\beta(P_w^\bullet)^\bullet}^\bullet\notin K^b(\text{pro} \ A)$, then $\ker\partial_w^0$ has an infinite minimal projective resolution. If $\ker\partial_w^0$ is as in (i), then at least one direct summand also has an infinite minimal projective resolution. If it is the case for $\ker p(w_1)$, it can be shown that there is an arrow $a\in Q_c^*$ such that $a\cdot w\in GSt$, and hence $w\in \overline{GSt}_c^*$. 
	
	If the minimal projective resolution of $\ker p(w_{l}^{-1})\cap\ker p(w_{l+1})$ is not bounded for some $2\leq l\leq n-1$ with $\mu_w(l)=0$, the same arguments allow us to show that there exists $a\in Q_c^*$ such that $a\in\ker p(w_{l}^{-1})\cap\ker p(w_{l+1})$. Therefore, the conclusion follows.\\
	
	If $\ker\partial_w^0$ is as in (iii), the reasoning is the same. Finally, if $\ker\partial_w^0$ is as in (ii) and $\ker p(w_{n}^{-1})$ has infinite minimal projective resolution, we consider $w^{-1}$ instead of $w$ and it can be shown that $w\cong_s w^{-1}\in\overline{GSt}_c^*$.
\end{proof}

If we put  Lemmas \ref{lemma7-n-converse-SAG} and \ref{Lemma7-n-direct-SAG} together we have the following result for SAG algebras, which provides a characterization for a string complex to have infinite minimal projective resolution. A complex with this property will be called a \textit{periodic string complex}.  
\begin{theorem}\label{Lemma7-SAG}
	Let $A=k(Q,I)$ be a SAG algebra. Then
	\begin{enumerate}
		\item For every $w\in GBa$ and $f\in \text{Ind} \ k[x]$, we have $\beta(P_{w,f}^\bullet)=P_{w,f}^\bullet$.
		\item If $w=w_1\cdot w_2\cdots w_n$ is a generalized string, then $P_{\beta(P_w^\bullet)^\bullet}^\bullet\notin K^b(\text{pro} \ A)$ if and only if either $w\in\overline{GSt}_c^*$ or $w\in\overline{GSt}_*^c$.
	\end{enumerate}
\end{theorem}

\subsection{Global dimension of a SAG algebra}

\vspace{1cm}

Let $A$ be a finite dimensional $k$-algebra. Recall that given an $A$-module $M$, the \textit{projective dimension} of $M$, denoted by $pd \ M$ is the smallest integer $d$ such that there exists a projective resolution of the form

$$\xymatrix{0\ar[r]&P_d\ar[r]&P_{d-1}\ar[r]&\cdots\ar[r]&P_1\ar[r]&P_0\ar[r]&M}.$$

If no resolution exists, then we say that $M$ has infinite projective dimension. The \textit{global dimension of $A$}, denoted by $gl.dim A$ is defined as the supremum of the projective dimensions of all $A$-modules, that is, $$gl.dim A:=\text{sup}\{pd \ M \mid M\in A-\text{mod}\}.$$

Theorem \ref{Lemma7-SAG} is very important because it deals with the structure of a SAG algebra. That is, it establishes when a SAG algebra has infinite global dimension. We state this in the following result.

\begin{theorem}\label{Th:gl.dim A-SAG}
	Let $A=k(Q, I)$ be a SAG algebra. If $\overline{GSt}_c^*\neq\emptyset$ or $\overline{GSt}_*^c\neq\emptyset$, then $gl.dim A=\infty$.
\end{theorem}

\begin{proof}
	Suppose that $\overline{GSt}_c^*\neq\emptyset$. The case $\overline{GSt}_*^c\neq\emptyset$ is similar.
	If $w=w_1\cdot w_2\cdots w_n\in \overline{GSt}_c^*$, then $\mu(w)=0$ and there exists $a\in Q_c^*$ such that $a\cdot w=a\cdot w_1\cdot w_2\cdots w_n\in GSt$ or  $a\in\ker p(w_{l}^{-1})\cap\ker p(w_{l+1})$ for some even index $l> 0$ with $\mu_w(l)=0$.
	
	Again, we will consider the case where $a\cdot w\in GSt$. The case in which $a\in\ker p(w_{l}^{-1})\cap\ker p(w_{l+1})$ is completely analogous. If $a\cdot w\in GSt$ and $\mu(w)=0$, necessarily $w_1\in\textbf{Pa}_{>0}$ and hence $aw_1=0$. According to Lemma \ref{New-lemma5-part1}, the general structure of $\ker p(w_1)$ is $\ker p(w_1)=Aa\oplus Ab$, where $b$ is an arrow such that $t(b)=s(w_1)$.
	Following the proof of Lemma \ref{lemma7-n-converse-SAG}, since $a\in Q_c^*$, we get that $Aa$ has infinite minimal projective resolution and this implies that $gl.dim \ A=\infty$.
\end{proof}

\begin{example}\label{Example-Qc*} Let $Q$ be the bound quiver
	
	$$ \xymatrix{&1\ar[r]^a&2\\
		5\ar@(ul,dl)[]_{x} \ar[r]_d &3\ar[r]_c\ar[ur]^b&4}
	$$	
	with $I=\langle db, dc, x^2, xd\rangle$. Then $A=kQ/I$ is a string almost gentle algebra over $k$, which is not a gentle algebra. In this case we have that $Q_c=\{x\}$ but $Q_c^*=\{x,d,c,b\}$.  
	
	Now, consider the generalized string $w=w_1\cdot w_2\cdot w_3=a\cdot b^{-1}\cdot c$. Then, $l(w)=3>0$, $\mu(w)=0$ and $\exists d\in Q_c^*$ such that $d\in \ker p(w_{2}^{-1})\cap\ker p(w_3)=\ker p(b)\cap \ker p(c)$, with $\mu_w(2)=0$. That is, $w\in \overline{GSt}_c^*$.
	
	According to Definition \ref{def-string-complex}, the complex $P_w^\bullet$ is 
	$$P_w^\bullet: \ \xymatrix{\cdots\ar[r]&0\ar[r]&P_1\oplus P_3\ar[r]^{\partial_w^0}&P_2\oplus P_4\ar[r]&0\ar[r]&\cdots}$$ 
	
	where $\partial_w^0=\begin{pmatrix}
	p(a)  & 0\\
	p(b)  & p(c)
	\end{pmatrix}$        
	
	Now, $\begin{pmatrix}u & v\end{pmatrix}\in\ker \partial_w^0$ if and only if $ua+vb=0$ and $vc=0$. Since $u\in P_1$ and $v\in P_3$, then $u=\alpha_0e_1$ and $v=\beta_0e_3+\beta_1d$ for some scalars $\alpha_0, \beta_0, \beta_1\in k$.
	
	From $ua+vb=0$, it follows that $\alpha_0a+\beta_0b=0$, and hence $\alpha_0=\beta_0=0$. Notice that the condition $vc=0$ is already satisfied. Therefore, we have that $u=0$ and $v=\beta_1d\in Ad=\ker p(b)\cap\ker p(c)$ and thus $$\ker \partial_w^0\cong \ker p(b)\cap\ker p(c)=Ad$$ (in general, we have for this case that $\ker\partial_w^0=\ker p(w_1)\oplus\ker p(w_2^{-1})\cap\ker p(w_3)$).	
	
	The projective cover of this kernel is $p(d): P_5\longrightarrow Ad$ and $\ker p(d)=Ax$. Since $x^2=0$, it is clear that the minimal projective resolution $P_{\beta(P_w^\bullet)^\bullet}^\bullet$ of $P_w^\bullet$ is
	
	$$\xymatrix@C=4mm@R=4.7mm{\cdots\ar[rr]&&P_5\ar[rr]^{p(x)}\ar@{->>}[dr]&                              &P_5\ar[rr]^{p(d)}\ar@{->>}[dr]&                           &P_w^0\ar[rr]^{\partial_w^0}&&P_w^1\ar[rr]&&0\ar[r]&\cdots\\
		&&                                                             &Ax\ar@{^(->}[ru]&                                                    & Ad\ar@{^(->}[ru]&                                       &&                       && }
	$$
	that is, $P_{\beta(P_w^\bullet)^\bullet}^\bullet\notin K^b(\text{pro} \ A)$.       
	
	Notice also that, since the generalized string $w=w_1\cdot w_2\cdot w_3=a\cdot b^{-1}\cdot c \in\overline{GSt}_c^*$, then, according to the previous theorem, we have that $gl.dim A=\infty$.         
\end{example}

\section{The Main Theorem For SAG Algebras}\label{Sec:main-theorem-SAG}
In this section, we will give a combinatorial description of a family of indecomposable objects in the bounded derived categories of SAG algebras and we will extend this result for some classes of string algebras. We begin with some technical lemmas in which we will follow a similar reasoning as in \cite{Be-Me}  and \cite{Gi-Ve} giving the details in an explicit way.

For the next lemma, recall that a nontrivial path $w$ of $Q$ is in $\textbf{Pa}$ if and only if it is a sub-path of a maximal path $\widetilde{w}$ that is not in $I$ (that is, of an element of $\textbf{M}$). This maximal path has the form $\widetilde{w}=\hat{w}w\bar{w}$ with $\hat{w}, \bar{w}\in \textbf{Pa}$ and, according to Lemma \ref{lem-3.1}, this maximal path is unique for SAG algebras.

The proofs of Lemmas \ref{Lemma-rho} and \ref{Lemma-image-of-rho} below follow the same structure and methods as in step 1 of the proof of Theorem 3 in \cite{Be-Me}.
\begin{lemma}\label{Lemma-rho}
	Let $A=kQ/I$ be a SAG algebra and let $w=w_1\cdot w_2\cdots w_n$ be a generalized string. Then the map $$\rho:\overline{GSt(A)}\times \mathbb{Z}\longrightarrow \overline{St(\mathcal{Y}(A))}$$ defined by $$\rho(w,m):=\rho(w_1,m)\cdots\rho(w_n,m)$$ where $$\rho(w_i,m):=\left([\hat{w_i},\mu_w(i-1)+m],[\hat{w_i}w_i,\mu_w(i)+m]\right) \quad \text{if} \quad w_i\in\textbf{Pa}_{>0}$$ or $$\rho(w_i,m):=\left([\widehat{w_{i}^{-1}},\mu_w(i)+m],[\widehat{w_{i}^{-1}}w_{i}^{-1},\mu_w(i-1)+m]\right)^{-1}\quad\text{if}\quad w_{i}^{-1}\in\textbf{Pa}_{>0}$$ is well defined and injective.
\end{lemma}

\begin{proof}
	We first show that $\rho(w,m)\in\overline{St(\mathcal{Y}(A))}$, i.e. we must show that $t(\rho(w_i,m))=s(\rho(w_{i+1},m))$ and $p_2(\rho(w_i,m))\neq p_1(\rho(w_{i+1},m))$ for $1\leq i \leq n-1$. There are four cases.
	\begin{enumerate}
		\item $w_i, w_{i+1}\in\textbf{Pa}_{>0}$. 
		
		Since $t(\hat{w_i}w_i)=t(w_i)=s(w_{i+1})=t(\widehat{w_{i+1}})$, then $t(\rho(w_i,m))=\overline{[\hat{w_i}w_i,\mu_w(i)+m]}=\overline{[\widehat{w_{i+1}},\mu_w(i)+m]}=s(\rho(w_{i+1},m))$.
		
		On the other hand, since $\widehat{w_{i+1}}w_{i+1}\neq 0$ and $w_iw_{i+1}=0$ (because $w\in\overline{GSt(A)}$), we have $\hat{w_i}w_i\neq \widehat{w_{i+1}}$ and hence $p_2(\rho(w_i,m))=[\hat{w_i}w_i,\mu_w(i)+m]\neq [\widehat{w_{i+1}},\mu_w(i)+m]= p_1(\rho(w_{i+1},m))$.
		
		\item $w_i, w_{i+1}^{-1}\in\textbf{Pa}_{>0}$.
		
		In this case, since $t(\hat{w_i}w_i)=t(w_i)=s(w_{i+1})=t(w_{i+1}^{-1})=t(\widehat{w_{i+1}^{-1}}w_{i+1}^{-1})$, it follows that $t(\rho(w_i,m))=\overline{[\hat{w_i}w_i,\mu_w(i)+m]}=\overline{[\widehat{w_{i+1}^{-1}}w_{i+1}^{-1},\mu_w(i)+m]}=s(\rho(w_{i+1},m))$.
		
		On the other hand, since in this case $w_iw_{i+1}\in St$, we have $w_{i+1}^{-1}\neq w_i$ and therefore $\hat{w_i}w_i\neq \widehat{w_{i+1}^{-1}}w_{i+1}^{-1}$. This implies that $p_2(\rho(w_i,m))=[\hat{w_i}w_i,\mu_w(i)+m]\neq [\widehat{w_{i+1}^{-1}}w_{i+1}^{-1},\mu_w(i)+m]=p_1(\rho(w_{i+1},m))$.
		
		\item $w_{i}^{-1}, w_{i+1}\in\textbf{Pa}_{>0}$.
		
		Since $t(\widehat{w_{i}^{-1}})=s(w_{i}^{-1})=t(w_i)=s(w_{i+1})=t(\widehat{w_{i+1}})$, then $t(\rho(w_i,m))=\overline{[\widehat{w_{i}^{-1}},\mu_w(i)+m]}=\overline{[\widehat{w_{i+1}},\mu_w(i)+m]}=s(\rho(w_{i+1},m))$.
		
		On the other hand, since $w_iw_{i+1}\in St$, then $w_{i+1}\neq w_i^{-1}$ and since $A$ is string almost gentle, by Lemma \ref{lem-3.1}, we have  $\widehat{w_{i}^{-1}}\neq \widehat{w_{i+1}}$. It follows that $p_2(\rho(w_i,m))=[\widehat{w_{i}^{-1}},\mu_w(i)+m]\neq [\widehat{w_{i+1}},\mu_w(i)+m]=p_1(\rho(w_{i+1},m))$.
		
		\item $w_{i}^{-1}, w_{i+1}^{-1}\in\textbf{Pa}_{>0}$.
		
		In this case, since $t(\widehat{w_{i}^{-1}})=s(w_i^{-1})=t(w_i)=s(w_{i+1})=t(w_{i+1}^{-1})=t(\widehat{w_{i+1}^{-1}}w_{i+1}^{-1})$, then $t(\rho(w_i,m))=\overline{[\widehat{w_{i}^{-1}},\mu_w(i)+m]}=\overline{[\widehat{w_{i+1}^{-1}}w_{i+1}^{-1},\mu_w(i)+m]}=s(\rho(w_{i+1},m))$.
		
		On the other hand, we know that $w_{i+1}^{-1}w_i^{-1}=0$ (because $w\in\overline{GSt(A)}$) and since $\widehat{w_{i}^{-1}}w_i^{-1}\neq 0$ then $\widehat{w_{i}^{-1}}\neq \widehat{w_{i+1}^{-1}}w_{i+1}^{-1}$. Hence $p_2(\rho(w_i,m))=[\widehat{w_{i}^{-1}},\mu_w(i)+m]\neq [\widehat{w_{i+1}^{-1}}w_{i+1}^{-1},\mu_w(i)+m]=p_1(\rho(w_{i+1},m))$.
	\end{enumerate}
	Now we show that $\rho$ is injective. First, we observe that, since $A$ is a string almost gentle algebra, for any $u,v\in\textbf{Pa}_{>0}$, by Lemma \ref{lem-3.1}, it follows that $\hat{u}=\hat{v}$ and $\hat{u}u=\hat{v}v$ if and only if $u=v$.
	
	Suppose $\rho(w,m)=\rho(v,m')$ for generalized strings $w=w_1\cdot w_2\cdots w_n, v=v_1\cdot v_2 \cdots v_r$ and for integers $m,m'$. Then $\rho(w_1,m)\cdots\rho(w_n,m)=\rho(v_1,m')\cdots\rho(v_r,m')$ and hence $n=r$ and $\rho(w_i,m)=\rho(v_i,m')$ for $1\leq i \leq n$. Again we have four cases. 
	\begin{enumerate}
		\item $w_i, v_i\in\textbf{Pa}_{>0}$. 
		
		In this case we have $$\left([\hat{w_i},\mu_w(i-1)+m], [\hat{w_i}w_i,\mu_w(i)+m]\right)=\left([\hat{v_i},\mu_v(i-1)+m'], [\hat{v_i}v_i,\mu_v(i)+m']\right).$$ Then, $\hat{w_i}=\hat{v_i}$, $\hat{w_i}w_i=\hat{v_i}v_i$, $\mu_w(i-1)+m=\mu_v(i-1)+m'$ and $\mu_w(i)+m=\mu_v(i)+m'$. Therefore, $w_i=v_i$, $\mu_w=\mu_v$ and $m=m'$. Thus, $(w,m)=(v,m')$.
		
		\item $w_i, v_i^{-1}\in\textbf{Pa}_{>0}$.
		
		In this case we have $$\left([\hat{w_i},\mu_w(i-1)+m], [\hat{w_i}w_i,\mu_w(i)+m]\right)=\left([\widehat{v_i^{-1}},\mu_v(i)+m'], [\widehat{v_i^{-1}}v_i^{-1},\mu_v(i-1)+m']\right)^{-1},$$ which is a contradiction since we would have an arrow in $Q(\mathcal{Y}(A))_1$ equal to an inverse arrow.
		
		\item $w_i^{-1}, v_i\in\textbf{Pa}_{>0}$. 
		
		This case is similar to the previous one, i.e., it cannot occur.
		
		\item $w_i^{-1}, v_i^{-1}\in\textbf{Pa}_{>0}$.
		
		In this case we have $$\left([\widehat{w_{i}^{-1}},\mu_w(i)+m], [\widehat{w_{i}^{-1}}w_i^{-1},\mu_w(i-1)+m]\right)^{-1}=\left([\widehat{v_{i}^{-1}},\mu_v(i)+m'], [\widehat{v_{i}^{-1}}v_i^{-1},\mu_v(i-1)+m']\right)^{-1},$$ which implies that $\widehat{w_{i}^{-1}}=\widehat{v_{i}^{-1}}$, $\widehat{w_{i}^{-1}}w_i^{-1}=\widehat{v_{i}^{-1}}v_i^{-1}$, $\mu_w(i)+m=\mu_v(i)+m'$ and $\mu_w(i-1)+m=\mu_v(i-1)+m'$. Thus, $w_i^{-1}=v_i^{-1}$ and hence $w_i=v_i$. As before, we conclude that $(w,m)=(v,m')$.
	\end{enumerate}
	This shows that $\rho$ is an injective map.
\end{proof}

\begin{lemma}\label{Lemma-image-of-rho}
	With the notation used in the previous Lemma, we have $$\text{Im} \ \rho=\{u\in\overline{St(\mathcal{Y}(A))} \  |  \ B_u\in\text{Im} \ F \},$$ where $F:\mathfrak{p}(A)\longrightarrow s(\mathcal{Y}(A),k)$ is the functor of Definition \ref{def-functor}.
\end{lemma}

\begin{proof}
	Let $u=u_1\cdot u_2 \dots u_n\in\overline{St(\mathcal{Y}(A))}$ such that $B_u\in\text{Im} \  F$. Then, for some $x_i\in\textbf{Pa}_{>0}$, we have (see the definition of $B_u$ in subsection \ref{subsec-indecomposables-in-Bondarenko}) $$u_i=\left([\hat{x_i},\mu_u(i-1)],[\hat{x_i}x_i,\mu_u(i)]\right) \quad \text{if} \quad u_i\in Q(\mathcal{Y}(A))_1,$$ or $$u_i=\left([\hat{x_i},\mu_u(i)],[\hat{x_i}x_i,\mu_u(i-1)]\right)^{-1} \quad \text{if} \quad u_i^{-1} \in Q(\mathcal{Y}(A))_1.$$
	We set 
	
	$$y_i= \left\{ \begin{array}{lcc}
	x_i       &   \text{if}  & u_i\in Q(\mathcal{Y}(A))_1, \\
	x_i^{-1}&  \text{if}   & u_i^{-1}\in Q(\mathcal{Y}(A))_1.
	\end{array}
	\right.$$
	
	We will show that $y=y_1\cdot y_2\cdots y_n\in\overline{GSt(A)}$. More precisely, we will show that $y_i\cdot y_{i+1}\in\overline{GSt(A)}$ for $1\leq i<n$. We have four cases:
	\begin{enumerate}
		\item $u_i, u_{i+1}\in Q(\mathcal{Y}(A))_1$.
		
		In this case, since $t(u_i)=s(u_{i+1})$ then $\overline{[\hat{x_i}x_i,\mu_u(i)]}=\overline{[\widehat{x_{i+1}},\mu_u(i)]}$ and hence $t(\hat{x_i}x_i)=t(\widehat{x_{i+1}})$. Thus, $t(x_i)=t(\widehat{x_{i+1}})=s(x_{i+1})$. On the other hand, since $p_2(u_i)\neq p_1(u_{i+1})$, so that $\hat{x_i}x_i\neq \widehat{x_{i+1}}$, and since $\widehat{x_{i+1}}x_{i+1}\neq 0$, we have $x_ix_{i+1}=0$ because $A$ is string almost gentle. It follows that $y_i\cdot y_{i+1}\in\overline{GSt(A)}$.
		
		\item $u_i, u_{i+1}^{-1}\in Q(\mathcal{Y}(A))_1$.
		
		In this case, since $\overline{[\hat{x_i}x_i,\mu_u(i)]}=\overline{[\widehat{x_{i+1}}x_{i+1},\mu_u(i)]}$, then $t(\hat{x_i}x_i)=t(\widehat{x_{i+1}}x_{i+1})$ and hence $t(x_i)=t(x_{i+1})$. Now, since $A$ is a string almost gentle algebra and $p_2(u_i)\neq p_1(u_{i+1})$, so that $\hat{x_i}x_i\neq \widehat{x_{i+1}}x_{i+1}$, it follows from Corollary \ref{cor3.3} that $x_{i+1}^{-1}\neq x_i^{-1}$. Therefore, $x_ix_{i+1}^{-1}\in St$ and thus $y_i\cdot y_{i+1}\in\overline{GSt(A)}$.
		
		\item $u_i^{-1}, u_{i+1}\in Q(\mathcal{Y}(A))_1$.
		
		Here, $t(u_i)=s(u_{i+1})$ implies $\overline{[\hat{x_i},\mu_u(i)]}=\overline{[\widehat{x_{i+1}},\mu_u(i)]}$. It follows that $t(\hat{x_i})=t(\widehat{x_{i+1}})$ and so $s(x_i)=s(x_{i+1})$. Now, since $A$ is a string almost gentle  algebra and $p_2(u_i)\neq p_1(u_{i+1})$, so that $\hat{x_i}\neq \widehat{x_{i+1}}$, then by Lemma \ref{lem-3.1}, we have $x_{i+1}\neq x_i$. Therefore, $x_i^{-1}x_{i+1}\in St$ and $y_i\cdot y_{i+1}\in\overline{GSt(A)}$.
		
		\item $u_i^{-1}, u_{i+1}^{-1}\in Q(\mathcal{Y}(A))_1$.
		
		In this case, $t(u_i)=s(u_{i+1})$ implies $\overline{[\hat{x_i},\mu_u(i)]}=\overline{[\widehat{x_{i+1}}x_{i+1},\mu_u(i)]}$, so that $t(\hat{x_i})=t(\widehat{x_{i+1}}x_{i+1})$, whence $t(\hat{x_i})=t(x_{i+1})$ and $s(x_i)=t(x_{i+1})$. On the other hand, since $p_2(u_i)\neq p_1(u_{i+1})$, then $\hat{x_i}\neq \widehat{x_{i+1}}x_{i+1}$. It follows from Corollary \ref{cor3.3} that $\hat{x_i}\neq x_{i+1}$ and, since $A$ is a string almost gentle algebra and $\hat{x_i}x_i\neq 0$, we have that $x_{i+1}x_i=0$. Therefore, $y_{i+1}^{-1}y_i^{-1}=0$, that is, $y_i\cdot y_{i+1}\in\overline{GSt(A)}$.
	\end{enumerate}
	Finally, we show that $\rho(y,0)=u$. We know that $\rho(y,0)=\rho(y_1,0)\cdots\rho(y_n,0)$. If $u_i\in Q(\mathcal{Y}(A))_1$, then $y_i=x_i\in\textbf{Pa}_{>0}$ and $$\rho(y_i,0)=\left([\hat{x_i},\mu(i-1)],[\hat{x_i}x_i, \mu(i)]\right)=u_i.$$
	If $u_i^{-1}\in Q(\mathcal{Y}(A))_1$, then $y_i=x_i^{-1}$ and $y_i^{-1}=x_i\in\textbf{Pa}_{>0}$. Hence
	$$\rho(y_i,0)=\left([\hat{x_i},\mu(i)],[\hat{x_i}x_i, \mu(i-1)]\right)^{-1}=u_i.$$
	We have used the fact that $\mu_u=\mu_y=\mu$. Therefore, $\rho(y,0)=u$.

	We have shown that $$\{u\in\overline{St(\mathcal{Y}(A))} \  |  \ B_u\in\text{Im} \ F \}\subseteq \text{Im} \ \rho.$$
	
	By the definition of $\rho$, we have the other inclusion.
\end{proof}

\begin{remark}
	Let $\rho_b$ be the restriction of $\rho$ on $GBa$. Then it is clear that 
	$$\text{Im} \ \rho_b=\{v\in\overline{Ba(\mathcal{Y}(A))} \  |  \ B_{v,1}\in\text{Im} \ F \}.$$
\end{remark}

Now, we can use the previous lemmas in order to state and prove our main result, in which we give a combinatorial description of a family of indecomposable objects in the bounded derived category of a string almost gentle algebra. That is, the theorem tells us how to construct indecomposable object in $D^b(A)$ when $A$ is a SAG algebra.

\begin{theorem}\label{Thm-Main-Theorem-SAG}
	Let $A=kQ/I$ be a string almost gentle algebra. Then 
	\begin{eqnarray*}
		\text{ind}_0  D^b(A) &\supseteq&\{T^i(P_w^\bullet) \mid w\in GSt, i\in\mathbb{Z}\}\dot{\cup}\\
		&  &\{T^i(P_{w,f}^\bullet) \mid w\in GBa, f\in\text{Ind} \ k[x], i\in\mathbb{Z}\}\dot{\cup}\\
		&  &\{T^i(\beta(P_w^\bullet)^\bullet) \mid w\in \overline{GSt}_c^*, i\in\mathbb{Z}\}\dot{\cup}\\
		&  & \{T^i(\beta(P_w^\bullet)^\bullet) \mid w\in \overline{GSt}_*^c\setminus \overline{GSt}_c^*, i\in\mathbb{Z}\}.
	\end{eqnarray*}
	where $T$ is the translation functor.
\end{theorem}
%
\begin{proof}
	Let $u\in GSt$, $v\in GBa$, $m\in\mathbb{Z}$ and $f\in \text{Ind}k[x]$. Then, it follows from Lemmas \ref{Lemma-rho} and \ref{Lemma-image-of-rho} that $$F(P_u^\bullet)\cong B_{\rho(u,m)} \quad \text{and} \quad F(P_{v,f}^\bullet)\cong B_{\rho_b(v,m),f}$$
	Therefore, by Theorem \ref{Th-indecomposables-in-Bondarenko} (see also Proposition 1 in \cite{Bo} or Theorem 2 in \cite{Be-Me}) and  Lemma \ref{Lemma4-SAG}, it follows that
	
	\begin{eqnarray*}
		\{T^i(P_u^\bullet) \mid u\in GSt, i\in\mathbb{Z}\} &&\\
		\dot{\cup}	\{T^i(P_{v,f}^\bullet) \mid v\in GBa, f\in\text{Ind} \ k[x], i\in\mathbb{Z}\} &\subseteq & \text{ind}_0 \ K^b(\text{pro} \ A)
	\end{eqnarray*}
	
	On the other hand, from Theorem \ref{Lemma7-SAG} we have that 
	$$
	\begin{array}{l}
	\{\beta(M^\bullet)^\bullet \mid M^\bullet\in \text{ind}_0 \ \mathfrak{p}(A) \ \text{and} \ P_{\beta(M^\bullet)^\bullet}^\bullet\notin K^b(\text{pro} \ A)\} =\\
	\{T^i(\beta(P_w^\bullet)^\bullet) \mid w\in \overline{GSt}_c^*, i\in\mathbb{Z}\}\dot{\cup} \{T^i(\beta(P_w^\bullet)^\bullet) \mid w\in \overline{GSt}_*^c\setminus \overline{GSt}_c^*, i\in\mathbb{Z}\}. 
	\end{array}
	$$
	
	Finally, using Corollary \ref{Cor:corollary1} the conclusion follows.
\end{proof}

\begin{remark}
	Example \ref{Ex-counter-example-to-lemma4} shows that the inclusion given in Theorem \ref{Thm-Main-Theorem-SAG} is a proper inclusion. That is, in the bounded derived category  $D^b(A)$, when $A$ is a string almost gentle algebra, there are indecomposable objects which are not string complexes, band complexes or periodic string complexes. Indeed, in Example \ref{Ex-counter-example-to-lemma4}, which is the simplest case for a SAG algebra, the complement $ind_0 D^b(A)\setminus\mathcal{F}$, where $\mathcal{F}$ is the family of indecomposable objects described in Theorem \ref{Thm-Main-Theorem-SAG}, is difficult to study. The reason for this is that,  according to Theorem 3.1 in \cite{Be-Dro}, this algebra is derived wild, in contrast to the gentle case, where all algebras are derived tame (see Theorem 4 in \cite{Be-Me}).
\end{remark}

To conclude this section, we will show that Theorem \ref{Thm-Main-Theorem-SAG} can be extended for some classes of string algebras. Some of the examples below (Example \ref{Ex-maximal-not-unique} and Example \ref{Ex-string-unique-maximal} (2)) are also presented in \cite{FGR}. Firstly, we emphasize the fact that, in the general case of string algebras, maximal paths are not unique, as it is shown in the following example.

\begin{example}\label{Ex-maximal-not-unique} 
	Consider the bound quiver $(Q, I)$ where $$Q: \xymatrix{1\ar[r]^a&2\ar[r]^b&3\ar[r]^c&4}$$ and $I=\langle abc\rangle$. Then $A=kQ/I$ is a string algebra which is neither a gentle algebra nor a SAG algebra.
	In this case the set of maximal paths is $\textbf{M}=\{ab, bc\}$. Thus, the arrow $b$ belongs to two different maximal paths, and so maximal paths are not unique for a given arrow.
\end{example}

According to this example, the fact that two different maximal paths have a common arrow is a consequence of allowing relations of length greater than two. In general, if there is a relation of the form $a_1a_2\cdots a_r=0$, with $r\geq 3$ and $t(a_r)\neq  s(a_1)$, then, maximal paths are not unique. This cannot occur in the case of SAG algebras in which relations always have length two. 

However, there are cases of string algebras for which maximal paths are unique as in the following examples.

\begin{example}\label{Ex-string-unique-maximal}
	\begin{enumerate}
		\item Let $(Q, I)$ be the bound quiver
		$$ \xymatrix@C=4.5mm{              & 3\ar[dl]_c&                &\\
			1\ar[rr]_a&                & 2\ar[ul]_b&}
		$$
		with $I=\langle abc\rangle$. Then $A=kQ/I$ is a string algebra over $k$, which is neither a gentle algebra nor a SAG algebra. In this case, $\textbf{M}=\{bcab\}$. Thus, maximal paths are unique for each arrow (trivially).
		\item  Consider the bound quiver $(Q, I)$
		$$ \xymatrix@C=4.5mm{ & 3\ar[dl]_c&                &&&6\ar[ld]_{c'}\\
			1\ar[rr]_a&& 2\ar[rr]_d\ar[ul]_b&&4\ar[rr]_{a'}&&5\ar[ul]_{b'}}
		$$
		where $I=\langle abc, ad, da', a'b'c'\rangle$. Then $A=kQ/I$ is a string algebra over $k$ that is neither gentle nor SAG. The set of maximal paths is $\textbf{M}=\{bcab, d, b'c'a'b'\}$. Thus, every arrow belongs to a unique maximal path.
		\item Let $(Q, I)$ be the bound quiver
		$$ \xymatrix@C=4.5mm{              & 3\ar[dl]_c\ar[rr]^e&                                                                  &4\ar[dl]^f\\
			                                   1\ar[rr]_a&                              & 2\ar@<-0.5ex>[ul]_d\ar@<0.7ex>[ul]^{b}&}
		$$
		with $I=\langle abc, def,be,dc,ad,fb\rangle$. Then $A=kQ/I$ is a string algebra over $k$, which is neither a gentle algebra nor a SAG algebra. In this case, $\textbf{M}=\{bcab, efde\}$. Thus, every arrow belongs to a unique maximal path.
	\end{enumerate}
\end{example}

According to the previous examples  we conjecture that if a string algebra $A=kQ/I$ has the property that,  \textit{if $I$ has relations of the form $a_1a_2\cdots a_r$ with $r\geq 3$ implies $t(a_r)=s(a_1)$}, then every arrow will be in a unique maximal path. Indeed, Example \ref{Ex-string-unique-maximal} shows that the class of string algebras with this property is nonempty.

For string algebras with the property that every arrow belongs to a unique maximal path, as in the previous examples, the construction of the functor $F: \mathfrak{p}(A)\longrightarrow s(\mathcal{Y}(A), k)$ made in subsection \ref{Subsec: the functor} remains valid for these cases. Thus, we have the following theorem, whose proof is similar to that of Theorem \ref{Thm-Main-Theorem-SAG}. It states that, for string algebras with unique maximal paths, string and band complexes are indecomposable objects in the corresponding derived categories.

\begin{theorem}\label{Th-indecomposables-string-case}
	Let $A=k(Q,I)$ be a string algebra with the property that every arrow belongs to a unique maximal path. Then
	\begin{eqnarray*}
		\text{ind}_0  D^b(A) &\supseteq&\{T^i(P_w^\bullet) \mid w\in GSt, i\in\mathbb{Z}\}\dot{\cup}\\
		&  &\{T^i(P_{w,f}^\bullet) \mid w\in GBa, f\in\text{Ind} \ k[x], i\in\mathbb{Z}\}	
	\end{eqnarray*}
	where $T$ is the translation functor.
\end{theorem}

\section{Periodic String Complexes over String Algebras}\label{Sec:periodic-complexes-for-string-case}

In this section, we will give a necessary and sufficient condition for a string complex to have infinite minimal projective resolution when the algebra is string. As before, these complexes are called \textit{periodic string complexes.} A very important consequence of this characterization is that we will obtain a sufficient condition for a string algebra to have infinite global dimension, as we did for the SAG case.
Some results in this section are presented without proofs (Lemmas \ref{lemma7-n-converse} and \ref{Lemma7-n-direct} and Theorems \ref{Lemma7-String} and \ref{Th:gl.dim A-String}), since most of the reasoning in these proofs are similar in structure to the SAG case. The reader interested in the detailed proofs is referred to \cite{FGR}. 

We begin with a preliminary result, which establishes that, despite maximal paths are not unique for string algebras, the left completion of an arrow $a$ (see section \ref{Sec-Preliminaries}), denoted by $\hat{a}$ is unique.
\begin{lemma}\label{lem-left-completion}
	Let $A=kQ/I$ be a string algebra and let $a\in Q_1$. Then $\hat{a}$ is unique.
\end{lemma}

\begin{proof}
	Suppose there are two different left completions $\hat{a}$ and $\hat{b}$ of $a$. 
	
	$$
	\xymatrix@R=5mm{\ar@{~>}[rrd]^{\hat{a}}&&&\\
		&&\ar[r]^a&\\
		\ar@{~>}[rru]_{\hat{b}}&&           &}
	$$
	Since $\hat{a}a$ and $\hat{b}a$ are elements in $\textbf{Pa}_{>0}$, then we have two compositions that do not lie in $I$, which is a contradiction with the definition of a string algebra.
\end{proof}
Now, we will study the structure of the kernel of a morphism $p(w)$ for some $w\in\textbf{Pa}_{>0}$ when $A=kQ/I$ is a string algebra. From the definition of this class of algebras, we know that there are at most two arrows $a$ and $b$ such that $t(a)=t(b)=s(w)$. We also know that it cannot happen that $aw\neq 0$ and $bw\neq 0$. Thus, at least one of $aw$ or $bw$ must be zero. Suppose, without loss of generality, that $bw=0$. For string algebras, it can occur that $aw\neq 0$ but there could be a path $p$ of the form $p=p_1\cdots p_r a$, where $p_1,\dots, p_r\in Q_1$, such that $pw=0$.  Let us denote by $a^*$ the smallest path of this form (in case it exists). Then, $a^*$ is the smallest subpath of $\hat{a}a$ with the property $a^*w=0$. 

This cannot happen if $A$ is gentle (resp. string almost gentle) and so, this is the fundamental difference with the case of gentle algebras (resp. string almost gentle algebras), in which the smallest paths with the mentioned property are arrows themselves. That is, in gentle algebras, if there are two arrows $a$ and $b$ ending at the vertex $s(w)$, then necessarily, $bw=0$, $aw\neq 0$ and so $b^*=b$. In SAG algebras it could happen that $aw=0$ and $bw=0$, and hence $a^*=a$ and $b^*=b$.

According to this notation, we have the following result, which is a generalization of Lemma 5 in \cite{Be-Me}. It allows us to study the structure of $\ker p(w)$ for some $w\in\textbf{Pa}_{>0}$, and more generally, the structure of the periodic string complexes for the case of string algebras.
\begin{lemma}\label{Lem-New-Lemma5-part1-String}
	Let $A=kQ/I$ be a string algebra and let $w\in\textbf{Pa}_{>0}$. Then, with the previous notation,  the structure of the kernel of $p(w)$ is $\ker p(w)=Aa^*\oplus Ab$.
\end{lemma}

\begin{proof}[Proof]
	
	Here we consider the case in which $a^*w=0$ and $bw=0$ (we do not exclude the possibility that $a^*=a$). It is clear that $Aa^*\oplus Ab \subseteq \ker p(w)$. For the other inclusion, let $u$ be an element of $\ker p(w)$. Then, since $A$ is a string algebra $u$ must have either $a$ or $b$ as its last arrow. In the latter case $u\in Ab$. In the former case, by the minimality of $a^*$, we have that $a^*$ is a sub-path of $u$ and hence $u\in Aa^*$ (here we are using the fact that both $u$ and $a^*$ are sub-paths of $\hat{a}a$, which is unique as a consequence of Lemma \ref{lem-left-completion}).
\end{proof}

At this point, we give a generalization of the cyclic sets $Q_c$, $\overline{GSt}_c$, $\overline{GSt}^c$, in order to characterize string complexes with infinite minimal projective resolution. 

Let us begin by defining a condition of minimality on paths in the kernel of a morphism $p(w)$. Notice that this condition is automatic for the cases of gentle and SAG algebras, in which the minimal generators of kernels are arrows.

\begin{definition}
	Let $p, w\in\textbf{Pa}_{>0}$ with $t(p)=s(w)$ and such that $pw=0$. We say that $p$ is \textit{minimal for $w$}  if no proper sub-path $p'$ of $p$ with $t(p')=s(w)$ verifies $p'w=0$.
\end{definition}
Thus, $a^*$ in the notation of  Lemma \ref{Lem-New-Lemma5-part1-String} satisfies this minimality condition, that is, $a^*$ is minimal for $w$.

Let us define a new set, which generalizes the sets of cyclic arrows $Q_c$ and $Q_c^*$. Let $\textbf{Pa}_c$ be the set of paths $p\in\textbf{Pa}_{>0}$ for which there exist paths   $p_m, p_{m-1}, \dots, p_1\in \textbf{Pa}_{>0}$ such that $t(p_m)=s(p)$, $t(p_i)=s(p_{i+1})$ for $i=1,\dots, m-1$, $s(p_1)=t(p_j)$ for some $1\leq j\leq m+1$, where $p_{m+1}=p$, with $p_ip_{i+1}=p_jp_1=0$. In addition, we require that $p_i$ is minimal for $p_{i+1}$ and $p_j$ is minimal for $p_1$.\\

It is clear that $Q_c\subseteq Q_c^*\subseteq \textbf{Pa}_c$, and, in the case of gentle algebras (resp. SAG algebras), we have $Q_c=\textbf{Pa}_c$ (resp. $Q_c^*=\textbf{Pa}_c$). The elements of $\textbf{Pa}_c$ are called \textit{cyclic paths}. 

%
Before we generalize the special sets $\overline{GSt}_c$ or $\overline{GSt}^c$, we consider the following example.

\begin{example} Let $(Q, I)$ be the bound quiver
	$$ \xymatrix@C=4.5mm{              & 3\ar[dl]_c&                &\\
		1\ar[rr]_a&                & 2\ar[ul]_b&}
	$$
	with $I=\langle ca, abc\rangle$. Then $A=kQ/I$ is a string algebra over $k$, which is neither  gentle nor SAG algebra. In this case, we have that $Q_c$ and $Q_c^*$ are empty sets but $\textbf{Pa}_c=\{a,bc,c,ab\}$ and $\textbf{Pa}_{>0}=\{a, b, c, ab, bc\}$.  
		
	Now, consider the generalized string $w=w_1=bc$. Then, $l(w)=1>0$, $\mu(w)=0$ and $\exists a\in\textbf{Pa}_c$ such that $a\cdot w=a\cdot bc\in GSt$. The complex $P_w^\bullet$ is 
	$$P_w^\bullet: \ \xymatrix{\cdots\ar[r]&0\ar[r]&P_2\ar[r]^{p(bc)}&P_1\ar[r]&0\ar[r]&\cdots}$$         
	
	Since $abc=0$ and $ca=0$, then $\ker p(bc)=Aa$, $\ker p(a)=Ac$, $\ker p(c)=Aab$ and $\ker p(ab)=Ac$. Thus, the (minimal) projective resolution $P_{\beta(P_w^\bullet)^\bullet}^\bullet$ is
	
	$$
	\xymatrix{\cdots\ar[r]&P_1\ar[r]^{p(ab)}&P_3\ar[r]^{p(c)}&P_1\ar[r]^{p(ab)}&P_3\ar[r]^{p(c)}&P_1\ar[r]^{p(a)}&P_2\ar[r]^{p(bc)}&P_1\ar[r]&0}
	$$
%
	Hence $P_{\beta(P_w^\bullet)^\bullet}^\bullet\notin K^b(\text{pro} \ A)$.

\end{example}
In this example, we observe that $a$ is minimal for $bc$ and when we calculate the projective resolution of $P_w^\bullet$, we get the cycle $\{c, ab\}$, which is a subset of $\textbf{Pa}_c$. The consequence of this is that $P_{\beta(P_w^\bullet)^\bullet}^\bullet\notin K^b(\text{pro} \ A)$. This motivates one of the conditions in the following generalization of the special sets.\\

We denote by $\overline{GSt}_{cp}$ the set of generalized strings $w=w_1\cdot w_2\cdots w_n$ of positive length such that $\mu(w)=0$, there exists $p\in\textbf{Pa}_c$  with $p\cdot w\in GSt$  or $p\in\ker p(w_{l}^{-1})\cap \ker p(w_{l+1})$ for some even index $l> 0$ with $\mu_w(l)=0$,  where $p$ is minimal for $w_1$ or minimal for $w_{l}^{-1}$ and $w_{l+1}$. \\

Also, we denote by $\overline{GSt}^{cp}$ the set of generalized strings $w=w_1\cdot w_2\cdots w_n$ of positive length such that $\mu(w)=\mu_w(n)$ and there exists $p\in \textbf{Pa}_c$ such that $w\cdot p^{-1}\in GSt$.\\

Thus, we have that these sets reduce to the corresponding one for the case of SAG algebras (resp. gentle algebras).

Using this new special sets, we can generalize the characterization of periodic string complexes given in Theorem \ref{Lemma7-SAG}, part 2. As we did before for the case of SAG algebras, we will divide this characterization into two lemmas whose proof is completely analogous to the corresponding one in Lemmas \ref{lemma7-n-converse-SAG} and \ref{Lemma7-n-direct-SAG}.

\begin{lemma}\label{lemma7-n-converse}
	Let $w=w_1\cdot w_2\cdots w_n$ be a generalized string. If $w\in \overline{GSt}_{cp}$ or $w\in \overline{GSt}^{cp}$, then $P_{\beta(P_w^\bullet)^\bullet}^\bullet\notin K^b(\text{pro} \ A).$
\end{lemma}

\begin{lemma}\label{Lemma7-n-direct}
	Let $w=w_1\cdot w_2\cdots w_n$ be a generalized string. If $P_{\beta(P_w^\bullet)^\bullet}^\bullet\notin K^b(\text{pro} \ A)$, then either $w\in\overline{GSt}_{cp}$ or $w\in\overline{GSt}^{cp}$.
\end{lemma}
Now, if put Lemmas \ref{lemma7-n-converse} and \ref{Lemma7-n-direct} together, we get the following theorem, which provides a characterization for a string complex to be periodic in the case of a string algebra.

\begin{theorem}\label{Lemma7-String}
	Let $A=k(Q, I)$ be a string algebra. Then
	\begin{enumerate}
		\item For every $w\in GBa$ and $f\in \text{Ind} \ k[x]$, we have $\beta(P_{w,f}^\bullet)=P_{w,f}^\bullet$.
		\item If $w=w_1\cdot w_2\cdots w_n$ is a generalized string, then $P_{\beta(P_w^\bullet)^\bullet}^\bullet\notin K^b(\text{pro} \ A)$ if and only if either $w\in\overline{GSt}_{cp}$ or $w\in\overline{GSt}^{cp}$.
	\end{enumerate}
\end{theorem}

Thus, a string complex $P_w^\bullet$ is periodic if and only if $w\in\overline{GSt}_{cp}$ or $w\in\overline{GSt}^{cp}$. 

\begin{remark}
	If we combine Theorem \ref{Lemma7-String} with Theorem \ref{Th-indecomposables-string-case}, we get an analogous of Theorem \ref{Thm-Main-Theorem-SAG} for the class of string algebras with the property that every arrow belongs to a unique maximal path. That is, if $A=k(Q,I)$ is a string algebra with this property, then
	\begin{eqnarray*}
		\text{ind}_0  D^b(A) &\supseteq&\{T^i(P_w^\bullet) \mid w\in GSt, i\in\mathbb{Z}\}\dot{\cup}\\
		&  &\{T^i(P_{w,f}^\bullet) \mid w\in GBa, f\in\text{Ind} \ k[x], i\in\mathbb{Z}\}	\\
		&  &\{T^i(\beta(P_{w}^\bullet)^\bullet) \mid w\in \overline{GSt}_{cp}, i\in\mathbb{Z}\}\dot{\cup}\\
		&  & \{T^i(\beta(P_{w}^\bullet)^\bullet) \mid w\in \overline{GSt}^{cp}\setminus \overline{GSt}_{cp}, i\in\mathbb{Z}\}.
	\end{eqnarray*}
	where $T$ is the translation functor (more details in \cite{FGR}).
\end{remark}
\subsection{Global dimension of a string algebra}

From the construction made in the previous section, particularly the proof of Lemmas \ref{lemma7-n-converse} and \ref{Lemma7-n-direct} (which is analogous to the proof of Lemmas \ref{lemma7-n-converse-SAG} and \ref{Lemma7-n-direct-SAG}, as we mentioned before), we obtain a sufficient condition for a string algebra to have infinite global dimension. We state this conditions as follows.

\begin{theorem}\label{Th:gl.dim A-String}
	Let $A=k(Q, I)$ be a string algebra. If $\overline{GSt}_{cp}\neq\emptyset$ or $\overline{GSt}^{cp}\neq\emptyset$, then $gl.dim A=\infty$. 
\end{theorem}

\begin{proof}
	Analogous to the proof of Theorem \ref{Th:gl.dim A-SAG}.
\end{proof}

\begin{remark}
An immediate conclusion from Theorems \ref{Th:gl.dim A-SAG} and \ref{Th:gl.dim A-String} is that the strong 
global dimension of $A$, $s.gl.dim A$, is infinite. Following \cite{Ha-Za1}, the strong global dimension of $A$ is
$$
s.gl.dim A=\text{sup}\left\{l(P^{\bullet})|\,P^{\bullet}\in\text{ind} \ K^b(\text{pro} \ A)\right\}, 
$$
where $l(P^{\bullet})$ is the length of the (indecomposable) complex $P^{\bullet}$. It is easy to check that 
$s.gl.dim A=\infty$ if $gl.dim A=\infty$. We suspect that the converses of Theorems \ref{Th:gl.dim A-SAG} and 
\ref{Th:gl.dim A-String} are true, but we are not able to prove them. However, as the following example shows, one can have a SAG 
algebra (resp. a string algebra) with infinite strong global dimension, finite global dimension, and where the special sets are empty.
\end{remark}

\begin{example}
	Let $(Q, I)$ be the bound quiver
	$$ \xymatrix@C=4.5mm{                           & 2\ar[dr]^b&   &\\
	                                      1\ar[rr]_c\ar[ru]^a&                & 3  &}
	$$
	with $I=\langle ab\rangle$. Then $A=kQ/I$ is a gentle (and therefore SAG and string) algebra over $k$. It is easy to see that $gl.dim A=2$, $s.gl.dim A=\infty$, and the special sets are empty (since $Q_c=\emptyset$).
\end{example}

\subsection*{Acknowledgments}

We acknowledge the important collaboration and many very helpful comments and suggestions of Viktor Bekkert, which were given during the visit of the first author to the Departamento de  Matem\'atica of the  Universidade Federal de Minas Gerais, and also for his hospitality. 

The second and the third authors are grateful to Germ\'an Ben\'itez Monsalve and Jos\'e A. V\'elez-Marulanda for several useful discussions about this work.

\subsection*{Funding}

This research was supported by Beca Doctorado Nacional Colciencias (Convocatoria 647 de 2014), CODI (Universidad de Antioquia, U de A), and Colciencias-Ecopetrol (no.0266-2013).


\end{document}